\pdfoutput=1

\documentclass{amsart}
\usepackage{amsmath}
\usepackage{amssymb}
\usepackage{oldgerm}
\usepackage{latexsym}
\usepackage{amsmath}
\usepackage{amssymb}
\usepackage{delarray}
\newtheorem{thm}{Theorem}[section] 
\newtheorem{lem}[thm]{Lemma}
\newtheorem{prop}[thm]{Proposition}
\newtheorem{conj}[thm]{Conjecture}
\newtheorem{rem}[thm]{Remark}

\newtheorem{cor}[thm]{Corollary}

\numberwithin{equation}{section}

\begin{document} 
 
\def\vecttwo(#1;#2){\left(\begin{array} {c} #1 \\ #2
\end{array}\right)}

\def\vecttwobrack(#1;#2){\left[\left(\begin{array} {c} #1 \\ #2
\end{array}\right)\right]}

\def\vectthree(#1;#2;#3){\left(\begin{array} {c} #1 \\ #2 \\ #3
\end{array}\right)}

\def\vectfour(#1;#2;#3;#4){\left(\begin{array} {c} #1 \\ #2 \\ #3 \\
#4 \end{array}\right)}

\def\mattwo(#1;#2;#3;#4){\left(\begin{array}{cc} #1 & #2 \\ #3 &#4
\end{array}\right)}

\def\matthree(#1,#2,#3;#4,#5,#6;#7,#8,#9){\left(\begin{array}{ccc} 
#1 & #2 & #3 \\ 
#4 &#5  & #6 \\
#7 &#8 & #9
\end{array}\right)}

\newcommand{\Sp}{\operatorname{Sp}}
\newcommand{\GSp}{\operatorname{GSp}}
\newcommand{\SL}{\operatorname{SL}}
\newcommand{\GL}{\operatorname{GL}}
\newcommand{\diag}{\operatorname{diag}}
\newcommand{\Z}{{\bf Z}}
\newcommand{\Q}{{\bf Q}}
\newcommand{\C}{{\bf C}}
\newcommand{\R}{{\bf R}}
\newcommand{\N}{{\bf N}}
\newcommand{\Half}{{\bf H}}
\def\period{{\Omega}}
\newcommand{\wideF}{{\widetilde {\mathcal F}}}
\renewcommand{\thefootnote}{\fnsymbol{footnote}}

\title[Period of the Ikeda-Miyawaki lift]  
{Period of the  Ikeda-Miyawaki lift}
\author {Tomoyoshi  Ibukiyama} 
\address{Tomoyoshi Ibukiyama, Department of Mathematics, Graduate School of Science, Osaka University, 
Machikaneyama 1-1, Toyonaka, Osaka, 560-0043 Japan}
\email{ibukiyam@math.sci.osaka-u.ac.jp}
\author{Hidenori  Katsurada}
\address{Hidenori Katsurada, Department of Mathematics, Hokkaido University,
Kita 10, Nishi 8, Kita-Ku, Sapporo, Hokkaido, 060-0810, Japan and  Muroran Institute of Technology, Mizumoto 27-1, Muroran, Hokkaido, 050-8585 Japan} 
\email{hidenori@mmm.muroran-it.ac.jp}
\author{Hisashi  Kojima}
\address{Hisashi Kojima  Department of Mathematics, Graduate School of Science and Engineering,
Saitama University, Shimo-Okubo 255, Sakura-ku, Saitama City, 338-8570, Japan}
\email{hkojima@rimath.saitama-u.ac.jp}
\thanks{The first  author was  partially supported by JSPS KAKENHI Grant Number JP JP19K03424, JP20H00115, JP23K03031. The second  author was  partially supported by JSPS KAKENHI Grant Number JP2103152}
\makeatletter
\@namedef{subjclassname@2020}{%
  \textup{2020} Mathematics Subject Classification}
\date{\today}
\maketitle
\begin{abstract}
In this paper, first we give a weak version of Ikeda's conjecture  on the period of the Ikeda-Miyawaki lift. Next, 
we confirm this conjecture rigorously in some cases.
\end{abstract}
\section{Introduction}
In this paper, we investigate Ikeda's conjecture on the period of the Ikeda-Miyawaki lift.
Let $k,n$ be positive integers such that $k \equiv n+1 \mod 2$ and $k >n$.
For a Hecke eigenform $h$ in the Kohnen plus subspace of weight $k+1/2$ for $\varGamma_0(4)$, and a primitive form $g$ of weight $k+n+1$ for $SL_2(\mathbb Z)$, Ikeda \cite{Ik2} constructed a cusp form $\mathcal F_{h,g}$ of weight $k+n+1$ for $Sp_{2n+1}(\mathbb Z)$ with the following property:

If $\mathcal F_{h,g}$ is not identically $0$, it is a Hecke eigenform and its standard $L$-function $L(s,\mathcal F_{h,g},\mathrm{St})$ is expressed as
\[L(s,\mathcal F_{h,g},\mathrm{St})=L(s,g,\mathrm{St})\prod_{i=1}^n L(s+k+n-i,f),\]
where $f$ is the primitive form of weight $2k$ for $SL_2(\mathbb Z)$ corresponding to $h$ via the Shimura correspondence, and $L(s,f)$ is Hecke's $L$-function of $f$.
The existence of such a Hecke eigenform was conjectured by Miyawaki \cite{Miy}. Therefore we call $\mathcal F_{h,g}$ the Ikeda-Miyawaki lift of $h$ and $g$.
To ensure the non-vanishing of $\mathcal F_{h,g}$, Ikeda \cite{Ik2} proposed the following conjecture concerning the period (or the Petersson norm ) of  $\mathcal F_{h,g}$:
$$2^{(2n-1)k+4n}\dfrac {\displaystyle \langle  {\mathcal F}_{h,g}, {\mathcal F}_{h,g} \rangle}{\displaystyle \langle g, g \rangle \langle h, h \rangle}=\displaystyle \Lambda(k+n,\mathrm {St}(g)\otimes f)
\prod_{i=1}^{n}\tilde \xi(2i)\prod_{i=1}^{n-1}\tilde \Lambda(2i+1,f,\mathrm {St})$$
(cf. Conjecture \ref{conj.A}). As for the  definitions of several $L$-values in question, see (\ref{Riemann-zeta}), (\ref{St}), (\ref{St-Hecke}). 
Ikeda also confirmed the conjecture approximately in some cases. But, as far as we know, there is no examples that confirms this conjecture rigorously.

In this paper, first we deduce a weak version of the above conjecture. More precisely, put
\[C_{h,g}=\dfrac{\displaystyle \langle g, g \rangle \langle h, h \rangle} {\displaystyle \langle  {\mathcal F}_{h,g}, {\mathcal F}_{h,g} \rangle}\displaystyle \Lambda(k+n,\mathrm {St}(g)\otimes f)
\prod_{i=1}^{n}\tilde \xi(2i)\prod_{i=1}^{n-1}\tilde \Lambda(2i+1,f,\mathrm {St}).\]
Then we prove that $C_{h,g}$ belongs to $\mathbb Q(f)\mathbb Q(g)$, where $\mathbb Q(f)$ and $\mathbb Q(g)$
are the Hecke fields of $\mathbb Q(f)$ and $\mathbb Q(g)$, respectively (cf. Theorem \ref{thm.algebraicity-of-period1}). As a result
we prove the algebraicity of $\dfrac {\displaystyle \langle  {\mathcal F}_{h,g}, {\mathcal F}_{h,g} \rangle}{\displaystyle \langle g, \ g \rangle^2\langle f, \ f \rangle^{n-1}  \langle h, h \rangle \Omega_-(f)}$, where $\Omega_{-}(f)$ is Manin-Shimura period for $f$, which will be defined in Section 1 (cf. Theorem \ref{thm.algebraicity-of-period2}).
Our proof relies on the algebraicities of  critical values of the product of the Hecke $L$-functions and the triple product $L$-functions of  elliptic cusp forms and the standard $L$-functions of  Siegel cusp forms (cf. \cite{B-S},\cite{Miz1},\cite{Or}).

Next,  we confirm Ikeda's conjecture rigorously in some cases (cf. Theorems \ref{thm.main-example} and \ref{thm.main-example16}). The idea is to compute the special values of $L$-functions stated above rigorously. We have an algorithm for computing the special values of  the triple product $L$-functions of  elliptic cusp forms (cf. \cite[Theorem 4.9]{I-K-P-Y}). We also have an algorithm for computing the special vales of the product  of the Hecke $L$-functions elliptic cusp forms
(cf. Proposition \ref{prop.product-Hecke-L}). One of key ingredients in the proof of main examples is to compute the standard $L$-functions of  Siegel cusp forms for $Sp_3(\mathbb Z)$ using the pullback formula of the Siegel Eisenstein series. This method has been implicitly used in \cite[Section 5]{I-K-P-Y}. However, there we used the pull-back formula without differential operator. In this paper, we use the pullback formula with differential operator (cf. Theorem \ref{thm.special-pullback}), which allows us to compute the values in question more efficiently (cf. Remark \ref{rem.special-pullback}).

This paper is organized as follows. In Section 2, we introduce several $L$-functions and review their algebraicity. In Section 3, we review the Ikeda-Miyawaki lift and Ikeda's conjecture on its  period, and prove a weaker version of Ikeda's conjecture.
In Section 4, we review the pullback of the Siegel-Eisenstein series. Moreover,  we give an algorithm for computing the special value of the standard $L$-function $L(s,F,\mathrm,{St})$ at $s=k-3$ when  $F$ is a Hecke eigenform of weight $k+2$ for $Sp_3(\mathbb Z)$. In Section 5, we give an algorithm for computing the special value of the product of Hecke $L$-functions of primitive forms for $SL_2(\mathbb Z)$.
Finally, in Section 6, we confirm Ikeda's conjecture in some cases using the results in Sections 5, 6, and \cite[Theorem 4.9]{I-K-P-Y}.

We thank Cris Poor and David Yuen for valuable comments.

{\bf Notation.}  We denote by ${\mathbb Z},$ the ring of rational integers, and by ${\mathbb Q}, {\mathbb R},$ and ${\mathbb C} $ 
the fields of  rational numbers, real numbers, and complex numbers, respectively. Moreover for a prime number $p,$ 
we denote by ${\mathbb Q}_p$ and ${\mathbb Z}_p$ the field of $p$-adic numbers and the ring of $p$-adic integers, respectively. 
We also denote by $\mathrm {ord}_p$ the normalized additive valuation on ${\mathbb Q}_p$ and write ${\bf e}(z)=e^{2 \pi i z}$ for $ z\in \C$.  

For a
commutative ring $R$, we denote by $M_{mn}(R)$ the set of
$(m,n)$-matrices with entries in $R.$  In particular put $M_n(R)=M_{nn}(R).$ 
 For an $(m,n)$-matrix $X$ and an $(m,m)$-matrix
$A$, we write $A[X] = {}^t X A X,$ where $^t X$ denotes the
transpose of $X$.  
Put $\GL_m(R) = \{A \in M_m(R)  \ | \  \det A \in R^\times \}$, where $\det
A$ denotes the determinant of the square matrix $A$, and $R^\times$
denotes the unit group of $R.$  
 Let $S_n(R)$ denote
the set of symmetric matrices of degree $n$ with entries in
$R.$ Furthermore, for an integral domain $R$ of characteristic different 
from $2,$  let ${\mathcal H}_n(R)$
denote the set of half-integral matrices of degree $n$ over
$R,$ that is, ${\mathcal H}_n(R)$ is the set of symmetric
matrices of degree $n$ whose $(i,j)$-component belongs to
$R$ or $\frac12 R$ according as $i=j$ or not. 
For a subset $S$ of
$M_n(R)$ we denote by $S^{\mathrm {nd}}$ the subset of $S$
consisting of non-degenerate matrices.  In particular, if
$S$ is a subset of $S_n({\mathbb R})$ with ${\mathbb R}$ the field
of real numbers, we denote by $S_{>0}$ (resp. $S_{\ge 0}$) the subset of $S$
consisting of positive definite (resp. semi-positive definite) matrices.  Let $R'$ be a
subring of $R.$ Two symmetric matrices $A$ and $A'$ with
entries in $R$ are called equivalent over $R'$ and we write $A \stackrel{\sim}{{ }_{R'}} A'$ if there is
an element $X$ of $\GL_n(R')$ such that $A'=A[X].$ We also write $A \sim A'$ if there is no fear of confusion. 
For square matrices $X$ and $Y$ we write $X \bot Y = \mattwo(X;0;0;Y).$
For an $m \times n$ matrix,  $B=(b_{ij})$ and sequences ${\bf i}=(i_1,\ldots,i_r)$ and ${\bf j}= (j_1,\ldots,j_r)$ of integers such that $1 \le i_1, \ldots, i_r \le m, 1 \le j_1,\ldots,j_r \le n$, we put 
\[B\begin{pmatrix}{\bf i}   \\ {\bf j}\end{pmatrix}=  (b_{i_k,j_l})_{1 \le k,l \le r}.\]
Let $T=(t_{ij})_{1 \le i, j \le m}$ be a symmetric matrix of variables. We say that $P(T)$ is a polynomial in $T$ if it is a polynomial in $t_{ij} \ (1 \le i \le j \le m)$.

\section{Several L-values}

Put $J_n=\mattwo(0_n;-1_n;1_n;\ \ 0_n),$ where $1_n$ denotes the unit matrix of degree $n.$ For a subring $K$ 
of ${\mathbb R}$ put  
$$\GSp_n^+(K)=\{M \in \GL_{2n}(K)   \ | \  J_n[M]=\mu(M) J_n \text  { for  some } \mu(M)>0 \},$$
and 
$$\Sp_n(K)=\{M \in \GL_{2n}(K)   \ | \  J_n[M]= J_n \}.$$
Define the standard subgroup $\varGamma_0^{(n)}(N)$ by 
 $$\varGamma_0^{(n)}(N)=\{\mattwo(A;B;C;D) \in \Sp_n({\mathbb Z}) \ | \ C \equiv 0  \mod N \}.$$
 If $n=1,$ we drop the superscript~$(n).$ 
Let ${\mathbb H}_n$ be Siegel's
upper half-space of degree $n$. For a function $F$ on ${\mathbb H}_n,$ an integer $k$ and $g=\mattwo(A;B;C;D) \in \GSp_n^+({\mathbb R}),$ put
$j(g,Z)=\det (CZ+D),$ and 
$$ \left( F|_k g \right)(Z)= (\det (g))^{k/2} j(g,Z)^{-k}F\left((AZ+B)(CZ+D)^{-1}\right).$$
 Let $k$ be a positive integer.  
 We denote by $M_k(\Sp_n(\mathbb Z))$ or  $M_k^{\infty}(
 \Sp_n(\mathbb Z))$ the space of holomorphic or $C^{\infty}$-modular forms, respectively,  of weight $k$ for  $\Sp_n(\mathbb Z).$ 
  We denote by $S_k(\Sp_n(\mathbb Z))$ the vector subspace of $M_{k}(\Sp_n(\mathbb Z))$ consisting of cusp forms. A holomorphic modular form $F(Z)$ for  $\Sp_n(\mathbb Z)$ has the following Fourier expansion:
\[F(Z)=\sum_{A \in \mathcal H_n(\mathbb Z)_{\ge 0}} c_F(A){\bf e}(\mathrm {tr}(AZ)),\]
and, in particular, a cusp form  has the following Fourier expansion:
\[F(Z)=\sum_{A \in \mathcal H_n(\mathbb Z)_{> 0}} c_F(A){\bf e}(\mathrm {tr}(AZ)),\]
where $\mathrm{tr}$ denotes the trace of a matrix.
Let $F$ and $G$ be elements of $M_k(\Sp_n(\mathbb Z))$ and assume that either one of $F$ and $G$ is a cusp form. We then  define the Petersson scalar product $\langle F, G \rangle $ by
  $$\langle F, G \rangle =
\int_{\Sp_n(\mathbb Z) \backslash \mathbb H_n} F(Z)\overline{ G(Z)}
 (\det Y)^{k-n-1} dXdY \qquad (Z=X+\sqrt{-1}Y).$$ 
 Let $\tilde {\bf L}_n={\mathcal R}_{{\mathbb Q}}(\GSp_n^+({\mathbb Q}),\Sp_{n}({\mathbb Z}))$, 
$ {\bf L}_n^{\circ}={\mathcal R}_{{\mathbb Q}}(\Sp_n({\mathbb Q}),\Sp_{n}({\mathbb Z}))$ be the Hecke rings over ${\mathbb Q}$  
for the Hecke pairs $(\GSp_n^+({\mathbb Q}), \Sp_{n}({\mathbb Z}))$, 
  $(\Sp_n({\mathbb Q}),\Sp_{n}({\mathbb Z}))$,  respectively.  
 The Hecke ring $ {\bf L}_n^{\circ}$ is a subring of 
$\tilde {\bf L}_n.$ Let $T=\Sp_n({\Z})M \Sp_n({\Z}) \in \GSp_n^+({\mathbb Q})$ be an element of 
 $\tilde {\bf L}_n.$ Write $T$ as a disjoint union  $T=\bigcup_{g} \Sp_n({\mathbb Z})g$.   
 For $k \in \N$ and  
for $F \in M_k(\Sp_n({\mathbb Z}))$ we  define the Hecke operator $F|_kT$ 
as 
$$F|_kT=(\det g)^{k/2-(n+1)/2} \sum_{g} F|_kg.$$
We define the elements $T(p)$ and $T_j(p^2)$ of $\tilde {\bf L}_n$ in a standard way (cf. \cite{An}): 
$T(p)= \Sp_n(\Z) (1_n \bot p1_n) \Sp_n(\Z)$ and 
$T_j(p^2){=} \Sp_n(\Z) (1_{n-j} \bot p1_j \bot p^21_{n-j} \bot p1_j) \Sp_n(\Z).$  
For any subring of 
 $R \subseteq \tilde {\bf L}_n$,   we say that  a modular form $F \in M_{k}(\Sp_n({\mathbb Z}))$ is  a Hecke eigenform 
with respect to $R$ if
 $F$ is a common eigenfunction of all $T \in R.$ If $R=\tilde {\bf L}_n,$ we simply say that $F$ is  a Hecke eigenform.
\begin{rem}\label{rem.mult-one}
Let $k \ge n+1$. Then, by the multiplicity one property for $S_{k}(\Sp_n({\mathbb Z}))$ (cf. \cite{A-M-R},\cite{C-L}), a Hecke eigenform   with respect to $ {\bf L}_n^{\circ}$ is a Hecke eigenform also with respect to $\tilde {\bf L}_n$.
\end{rem}
 In this case, we denote by ${\mathbb Q}(F)$ the field generated over ${\mathbb Q}$ by all the Hecke eigenvalues of 
$T \in \tilde {\bf L}_n,$  and call it the Hecke field of $F.$ 
In this section we review several L-values of modular forms that appear in this article.  Let 
$$f(z)=\sum_{m=1}^{\infty} c_{f}(m){\bf e}(mz)$$
 be a primitive form in $S_{k}(\SL_2({\mathbb Z})).$   
For each prime $p$ let $\alpha_p=\alpha_{f,p}$ be a complex number such that 
$\alpha_{p}+\alpha_{p}^{-1}=p^{-k/2+1/2}c_{f}(p).$ 
We define
$$\displaystyle L(s,f)=\prod_p \left((1-\alpha_p p^{k/2-1/2-s})(1-\alpha_p^{-1} p^{k/2-1/2-s}) \right)^{-1}.$$
We write  
$\Gamma_{\mathbb C}(s)=2(2\pi)^{-s}\Gamma(s)$ and write $\Gamma_{\mathbb R}(s)=\pi^{-s/2}\Gamma(s/2)$ as usual. 
Put
\begin{equation}\label{Riemann-zeta}
\widetilde \xi(s)=\Gamma_{\mathbb C}(s)\zeta(s)
\end{equation}
 with $\zeta(s)$ is Riemann's zeta function.
For an integer $l$ satisfying $1 \le l \le k-1$,  put 
$$\Lambda(l,f)=\Gamma_{\mathbb C}(l)L(l,f).$$ 
Then there exist two (positive) real numbers $\period_{+}(f)$ and $\period_{-}(f)$ such that 
 we have (cf. \cite{Sh1}) 
\begin{equation}\label{Per}
{   \dfrac{\Lambda(l,f,\chi)}{\tau(\chi)\period_{j(l,\chi)}(f)}} \in {\mathbb Q}(f) 
\end{equation}
for any integer $1 \le l \le k-1$ and a primitive character $\chi$,
where $j=j(l,\chi)=+$ or $-$ according as $\chi(-1)(-1)^l=1$ or $-1$, and $\tau(\chi)$ is the Gauss sum of $\chi$. We note that $\period_{+}(f)$ and $\period_{-}(f)$ 
are determined by $f$ only up to constant multiples of ${\mathbb Q}(f)^{\times}.$
Fixing $\period_{\pm}(f)$, we define an algebraic part of $L(l,f,\chi)$ by 
\[
L_{alg}(l,f,\chi)=\frac{\Lambda(l,f,\chi)}{\tau(\chi)\period_{j}(f)}.
\]
If $\chi$ is the principal character, we simply denote it by $L_{alg}(l,f)$.
For two positive integers $ l_1,l_2 \le k -1$ and Dirichlet characters $\chi_1,\chi_2$ such that $(\chi_1(-1)\chi_2(-1))=^{l_1+l_2+1}$, we define $L_{alg}(l_1,l_2;f;\chi_1,\chi_2)$ as 
\begin{equation}\label{V1}
L_{alg}(l_1,l_2;f;\chi_1,\chi_2)=\dfrac{\Gamma_{\mathbb C}(l_1)\Gamma_{\mathbb C}(l_2) L(l_1,f,\chi_1)L(l_2,f,\chi_2)}{\sqrt{-1}^{l_1+l_2+1}\langle f,f \rangle \tau((\chi_1\chi_2)_0)}, 
\end{equation}
 where $(\chi_1\chi_2)_0$ is the primitive character associated with $\chi_1\chi_2$. Then it is known that  $L_{alg}(l_1.l_2,f;\chi_1,\chi_2)$ belongs to  ${\mathbb Q}(f)(\chi_1,\chi_2)$, where ${\mathbb Q}(f)(\chi_1,\chi_2)$ is the field generated over $\mathbb Q(f)$ by all the values of $\chi_1$ and $\chi_2$ (cf. \cite{Sh0}, \cite
{Za}).   If $\chi_1$ and $\chi_2$ are principal characters, we simply denote it by $L_{alg}(l_1,l_2;f)$. 
 Let $F$ be  a Hecke eigenform in $S_k(\Sp_{n}({\mathbb Z}))$ with respect to  $ {\bf L}_n^{\circ},$ and 
 let the $p$-Satake parameters of $F$ be 
 $\beta_1(p),\cdots,\beta_n(p)$.   
We then define the standard $L$-function by 
$$L(s,F,\mathrm {St})=\prod_p \{(1-p^{-s})\prod_{i=1}^n (1-\beta_i(p)p^{-s})(1-\beta_i(p)^{-1}p^{-s})\}^{-1}.$$
We put 
\begin{equation}\label{St}
\widetilde \Lambda(s,F,\mathrm {St})=\Gamma_{\mathbb C}(s)
\left( \prod_{i=1}^n \Gamma_{\mathbb C}(s+k-i)\right) L(s,F, \mathrm {St}).
\end{equation}
Let $\rho(n)=3$ if  $n \equiv 1 \mod 4$ and $n \ge 5$, and $\rho(n)=1$ otherwise.  
For a positive integer $m$ put  
 $$L_{alg}(m,F, \mathrm {St})= {\dfrac{\displaystyle \widetilde \Lambda(m,F,\mathrm {St})}{\displaystyle \langle F,F \rangle }}.$$
It is known that $L_{alg}(m,F, \mathrm {St})$ belongs to ${\mathbb Q}(F)$ if 
 $\rho(n) \le m \le k-n$ and $m \equiv n \mod 2$ and all the Fourier coefficients of $F$ belong to ${\mathbb Q}(F)$ (cf. \cite{Bo1}, \cite{Miz1}).
When $F$ is a Ikeda-Miyawaki lift, which we shall explain in Remark \ref{rem.algebraicity-Ikeda-Miyawaki}, the algebraicity holds also for $m=1$. 
For general $F$, it is expected that $L_{alg}(1,F, \mathrm {St})\in {\mathbb Q}(F)$ 
even when $n \equiv 1 \mod 4$ and $n \ge 5,$ 
but this has not been proved in general.  

Let $f_i(z)$ be a primitive Hecke eigenform in $S_{k_i}(\SL_2({\mathbb Z}))$ for $i=1,2,3.$ 
Then we define the triple product $L$-function  
$\displaystyle L(s,f_1 \otimes f_2 \otimes f_3)$ as
$$\displaystyle L(s,f_1 \otimes f_2 \otimes f_3)=\prod_{p}\prod_{i,j,l\in \{1,-1\} }
 \left(1-\alpha_{f_1,p} ^i\alpha_{f_2,p}^j \alpha_{f_3,p}^l \,p^{(k_1+k_2+k_3-3)/2-s} \right)^{-1}.$$  
This satisfies a functional equation for $s \rightarrow k_1+k_2+k_3-2-s$.  
Note that $L(s,f_1 \otimes f_2 \otimes f_3)$ is symmetric in the~$f_i$, so that we may assume that 
 $k_1 \ge k_2 \ge  k_3$ without loss of generality.  
In addition, we always assume  that $k_2+k_3 >k_1$ hereafter in this article;  
this case is  called the {\it balanced case}. 
Then there is an integer~$l$ satisfying  $(k_1+k_2+k_3)/2-1\le l \le  k_2+k_3-2$, and  for such $l$,  
we put 
\begin{align*}
&L_{alg}(l,f_1 \otimes f_2 \otimes f_3)=  \\
&{\dfrac{L(l,f_1 \otimes f_2 \otimes f_3) \Gamma_{\mathbb C}(l)\Gamma_{\mathbb C}(l-k_1+1)
\Gamma_{\mathbb C}(l-k_2+1)\Gamma_{\mathbb C}(l-k_3+1)}
{ \langle f_1,f_1 \rangle \langle f_2,f_2 \rangle  \langle f_3,f_3 \rangle}}.
\end{align*}
Then  $L_{alg}(l,f_1 \otimes f_2 \otimes f_3)$ belongs to ${\mathbb Q}(f_1){\mathbb Q}(f_2){\mathbb Q}(f_3)$
 (\cite{Ga}, \cite{Or}, \cite{Sat}) and is also symmetric in the~$f_i$.  
  An algorithm to compute $L_{alg}(l,f_1 \otimes f_2 \otimes f_3)$ has been given in \cite[Section 4]{I-K-P-Y}. 

Finally,  we define $\displaystyle L(s,\mathrm {St}(f_2)\otimes f_1)$ as 
$$\displaystyle L(s,\mathrm {St}(f_2)\otimes f_1)=
\prod_{p} \left(\prod_{i \in \{1,-1 \}} \prod_{j \in \{1,-1,0 \}} (1-\alpha_{f_1,p}^i \alpha_{f_2,p}^{2j}  
\,p^{k_1/2-1/2-s}) \right)^{-1},$$
and put
\begin{equation}\label{St-Hecke}
 \displaystyle \Lambda(s,\mathrm {St}(f_2)\otimes f_1)=\Gamma_{\mathbb C}(s)\Gamma_{\mathbb C}(s -k_1+k_2)
\Gamma_{\mathbb C}(s+ k_2-1 )L(s,\mathrm {St}(f_2)\otimes f_1).
\end{equation}
Moreover we fix the period $\period_{\pm}(f_{1})$ satisfying (\ref{Per}) and 
for a positive integer $l$ with $k_1/2\leq l \leq k_2-1$, put 
$$L_{alg}(l,\mathrm {St}(f_2)\otimes f_1)={ \dfrac{\displaystyle \Lambda(l,\mathrm {St}(f_2)\otimes f_1)}
{\displaystyle \langle f_2,f_2 \rangle^2 \period_{j'}(f_1) }}, $$ 
where $j'=+$ or $-$ according as  $l$ is odd or even. 
(We note that $L_{alg}(l,\mathrm {St}(f_2)\otimes f_1)$ would more properly be denoted as 
$L_{alg}(l,\mathrm {St}(f_2)\otimes f_1; \Omega_{j'}(f_1))$  since it does depend upon the choice of 
$\period_{\pm}(f_1)$, but we use the above short notation.)  
From the Euler product, we note that 
$$L(s+k_2-1,f_2 \otimes f_2 \otimes f_1)=L(s,\mathrm {St}(f_2) \otimes f_1)L(s,f_1)$$
and we have 
\begin{align}\label{V2}
&L_{alg}(l+k_2-1,f_2 \otimes f_2 \otimes f_1) \\
&=L_{alg}(l,\mathrm {St}(f_2) \otimes f_1)L_{alg}(l,f_1)
 \frac{\Omega_{+}(f_1)\Omega_{-}(f_1)}{\langle f_1,f_1\rangle}. \notag
\end{align}
We also note that $\Lambda(l,f_1) /\Omega_j(f_1)$ belongs to ${\mathbb Q}(f_1)^{\times}$ for  $k _1/2 +1 \le l \le  k_2-1,$ 
since we assumed $ k_1/2 +1 \le  k_2 \le k_1$ and 
the Euler product of $L(s,f_1)$ converges in this range.
Hence $L_{alg}(l,\mathrm {St}(f_2)\otimes f_1)$ belongs to  ${\mathbb Q}(f_1){\mathbb Q}(f_2)$ and non-zero for $k _1/2 +1 \le l \le  k_2-1$.  


\section{Ikeda-Miyawaki lift}
Throughout this section, fix positive integers  $k,n$  such that $k \equiv n+1 \mod 2$ 
and  $k >n$.  
We denote by $S_{k+1/2}^+(\varGamma_0(4))$ the Kohnen plus subspace of weight $k+1/2$ for $\varGamma_0(4)$.
For the detail of the Kohnen plus subspace, see \cite{Ko}. 
For two elements $h_1$ and $h_2$ of $S_{k+1/2}^+(\varGamma_0(4))$, we define the Petersson norm
$\langle h_1, \ h_2 \rangle$ as
\[\langle h_1, \ h_2\rangle =\frac{1}{6}\int_{\varGamma_0(4) \backslash \mathbb H_1} h_1(z)\overline{h_2(z)}y^{k-3/2}dx dy \quad  (z=x+\sqrt{-1}y).\]
Let $h$   be a Hecke eigenform in  $S_{k+1/2}^+(\varGamma_0(4))$  
with Fourier expansion 
$$h(z)=\sum_{m \in {\N}, \ (-1)^{k}m \equiv 0,1 \mod 4}c_h(m){\bf e}(mz)$$
 and let $g$ be a primitive form in $S_{k+n+1}(\SL_2({\mathbb Z}))$ with Fourier expansion 
 $$g(z)=\sum_{m=1}^{\infty} c_g(m){\bf e}(mz).$$
Moreover let 
$$f(z)=\sum_{m=1}^{\infty}c_f(m){\bf e}(mz)$$
be  the   primitive form in $S_{2k}(\SL_2({\mathbb Z}))$ corresponding to $h$ under the  Shimura correspondence.  
Let $\alpha_p \in {\mathbb C}$ be taken such that $\alpha_p+\alpha_p^{-1}=p^{-k+1/2} c_f(p).$ 
For $T \in {{\mathcal H}_{2n+2}(\mathbb Z)}_{> 0},$ define $c_{I_{2n+2}(h)}(T)$  as
$$c_{I_{2n+2}(h)}(T)=c_{h}(|{\mathfrak d}_T|)\, {\textfrak f}_T^{k-1/2} 
\prod_{p \mid {\mathfrak f}_T} \alpha_p^{-\mathrm {ord}_p({\mathfrak f}_T)} F_p(T,p^{-n-3/2}\alpha_p),$$
where ${\mathfrak d}_T$ is the discriminant of ${\mathbb Q}(\sqrt{(-1)^{n+1} \det (2T)}), 
\ {\mathfrak f}_T= \sqrt{ (-1)^{n+1}\det (2T)/{\mathfrak d}_T}$, 
and $F_p(T,X)$ is a  polynomial in $X$  with coefficients in ${\mathbb Q},$ which will be defined in Section 4.  
Define a Fourier series $I_{2n+2}(h)(Z)$ for $Z \in {\mathbb H}_{2n+2}$: 
$$I_{2n+2}(h)(Z)= \sum_{T \in {{\mathcal H}_{2n+2}(\mathbb Z)}_{> 0}} c_{I_{2n+2}(h)}(T){\bf e}(\mathrm {tr}(TZ)).$$
Then $I_{2n+2}(h)$ is a Hecke eigenform in  $S_{k+n+1}(\Sp_{2n+2}({\mathbb Z}) )$, 
see  \cite{Ik1}.  A proof that $I_{2n+2}(h)$ is a Hecke eigenform 
for the entire Hecke algebra ${\tilde {\bf \ L}}_{2n+2}$, 
not just the even part  $ {\bf \ L}^{\circ}_{2n+2}$, 
may be found in \cite{Kat22}.  
We call  $I_{2n+2}(h)$ the Duke-Imamo{\=g}lu-Ikeda lift of $h$ (or of $f$) to 
$S_{k+n+1}(\Sp_{2n+2}({\mathbb Z}) ).$
For $z \in {\mathbb H}_{2n+1}$ and $ w=x+iy \in {\mathbb H}_1, $  define
$$\displaystyle {\mathcal F}_{h,g}(z)=\int_{\SL_2({\mathbb Z}) \backslash {\mathbb H}_1} 
I_{2n+2}(h)(\mattwo(z;0;0;w)) \overline{g(w)} y^{k+n-1} dw.$$
Ikeda \cite{Ik2} showed the following:
\smallskip

{\it  If $\displaystyle {\mathcal F}_{h,g}(z)$ is not identically zero  in 
$S_{k+n+1}( \Sp_{2n+1}({\mathbb Z}))$, 
then it is a Hecke eigenform  with respect to ${\bf L}_{2n+1}^{\circ}$ and 
its standard $L$-function is given by 
\begin{align}\label{I-M}
L(s,{\mathcal F}_{h,g},\mathrm {St})
=L(s,g,\mathrm {St})\prod_{i=1}^{2n} L(s+k+n-i,f). 
\end{align}}
This is a part of Ikeda's results: the case $r=1$ in \cite{Ik2}. 
The existence of this type of Hecke eigenform was conjectured by Miyawaki \cite{Miy}; therefore, 
we call ${\mathcal F}_{h,g}$ 
the Ikeda-Miyawaki lift of $h$ and $g$ when ${\mathcal F}_{h,g}$ is not identically zero.
By Remark \ref{rem.mult-one}, it is  a Hecke eigenform also with respect to 
$\tilde {\bf L}_{2n+1}$. 

\smallskip

Ikeda proposed the following conjecture relating the Petersson norm of ${\mathcal F}_{h,g}$ 
to those of $f$ and $g$.  
\begin{conj}\label{conj.A}  (Ikeda \cite{Ik2})  
Let $k>n$ be positive integers with $k+n$ odd.  
Let $h\in S_{k+1/2}^+(\varGamma_0(4))$ be a Hecke eigenform corresponding to 
the primitive eigenform $f \in S_{2k}(\SL_2({\mathbb Z}))$ under the Shimura correspondence.  
Let $g \in S_{k+n+1}(\SL_2({\mathbb Z}))$ be a primitive eigenform.  Then we have
$$2^{(2n-1)k+4n}\dfrac {\displaystyle \langle  {\mathcal F}_{h,g}, {\mathcal F}_{h,g} \rangle}{\displaystyle \langle g, g \rangle \langle h, h \rangle}=\displaystyle \Lambda(k+n,\mathrm {St}(g)\otimes f)
\prod_{i=1}^{n}\tilde \xi(2i)\prod_{i=1}^{n-1}\tilde \Lambda(2i+1,f,\mathrm {St}).$$
\end{conj}
The expression in the above conjecture appears to be different from the one in \cite{Ik2} but 
it is the same since we have $\widetilde{\Lambda}(1,f,St)=2^{2k}\langle f,f \rangle$ (see \cite{Za}).
We consider the case that ${\mathcal F}_{h,g}$ is not identically zero. Then, the above conjecture can be rewritten as follows:
\begin{conj}\label{conj.A'} In additions to the notation and the assumption as in Conjecture \ref{conj.A}, assume that the 
Ikeda-Miyawaki lift ${\mathcal F}_{h,g}$ is not identically zero.  
 Put 
$$C_{h,g}=\dfrac{\displaystyle \langle g, g \rangle \langle h, h \rangle} {\displaystyle \langle  {\mathcal F}_{h,g}, {\mathcal F}_{h,g} \rangle}\displaystyle \Lambda(k+n,\mathrm {St}(g)\otimes f)
\prod_{i=1}^{n}\tilde \xi(2i)\prod_{i=1}^{n-1}\tilde \Lambda(2i+1,f,\mathrm {St}).$$
Then
$$C_{h,g}=2^{(2n-1)k+4n}.$$
\end{conj}	
Now we modify  ${\mathcal F}_{h,g}$ and put 
$$\widetilde {\mathcal F}_{h,g}=\langle g, g \rangle^{-1} {\mathcal F}_{h,g}.$$ 
Then, we note that we have 
\begin{align}\label{C}
C_{h,g}&=\dfrac{\displaystyle  \langle h, h \rangle} {\displaystyle \langle g, g \rangle \langle  \widetilde {\mathcal F}_{h,g}, \ \widetilde {\mathcal F}_{h,g} \rangle} \\
& \times  \Lambda(k+n,\mathrm {St}(g)\otimes f)
\prod_{i=1}^{n}\tilde \xi(2i)\prod_{i=1}^{n-1}\tilde \Lambda(2i+1,f,\mathrm {St}).\notag
\end{align}
Let $\{g_i \}_{i=1}^{d_2}$ be a basis of   $S_{k+n+1}(\SL_2({\mathbb Z}))$ consisting of Hecke eigenforms, then 
 from the definition of the Ikeda-Miyawaki lift, we easily have:  
$$I_{2n+2}(h)(\mattwo(z;0;0;w)) =\sum_{i=1}^{d_2} \widetilde {\mathcal F}_{h,g_i}(z)g_i(w).$$
We obtain
$$ I_{2n+2}(h)(\mattwo(z;0;0;w)) =\sum_{A>0, m>0}c_{2n+1,h}(A,m){\bf e}(\mathrm {tr}(Az)){\bf e}(mw),$$
where 
$$c_{2n+1,h}(A,m)=\sum_{r} c_{I_{2n+2}(h)}(\mattwo(A;r/2;{}^tr/2;m)).$$

By replacing the Hecke eigenform $h$ by a constant multiple of $h$, 
we can assume that every Fourier coefficient $c_h(m)$ belongs to ${\mathbb Q}(f)$ for any $m >0.$ 
If we assume so, then by definition, $c_{I_{2n+2}(h)}(T)$ belongs to
${\mathbb Q}(f)$ for any $T \in {{\mathcal H}_{2n+2}}_{> 0}$ and, therefore,  
so does $c_{2n+1,h}(A,m) $ for any $A \in  {{\mathcal H}_{2n+1}}_{> 0}$ and $m >0.$  
Hence, the $A$-th Fourier coefficient $c_{\widetilde {\mathcal F}_{h,g}}(A)$  of 
$\widetilde {\mathcal F}_{h,g}$ belongs to ${\mathbb Q}(f){\mathbb Q}(g)$ for any 
$A \in {{\mathcal H}_{2n+1}}_{> 0}.$  

The following theorem is a refinement of \cite[Corollary to Theorem 3.1]{I-K-P-Y}. 
\bigskip

\begin{thm}\label{thm.standard-L-of-IM-lift1}
We assume that $k>n$ and $n$ is odd. 
We fix a real number $\period_{\pm}(f)$ which satisfies {\rm  (\ref{Per})} of Section~2.   Then
\begin{align*}
&|c_{\widetilde {\mathcal F}_{h,g}}(A)|^2 L_{alg}(l, \widetilde {\mathcal F}_{h,g} , \mathrm {St})
=(-1)^nC_{h,g}{\dfrac{ | c_{{\widetilde {\mathcal F}}_{h,g}}(A)|^2\, \period_{-}(f)}{\langle h,h \rangle }} \\
\times  &  { \dfrac{L_{alg}(k+n,f) L_{alg}(l,g ,\mathrm {St}) \prod_{i=1}^{n} L_{alg}(l+k+n-2i+1,l+k+n-2i;f)}
{ L_{alg}(2k+2n, g \otimes g \otimes f)\prod_{i=1}^{n-1} L_{alg}(2i+1,f,\mathrm {St}) \prod_{i=1}^n \widetilde \xi(2i)}}
\end{align*}
 for any positive definite half-integral  matrix  $A$ of degree $2n+1$ and  
an odd integer~$l$ with  $1 \le l \le k-n$. 
\end{thm}
\begin{proof}
The assertion directly follows from (\ref{V1}),(\ref{V2}), (\ref{I-M}) and (\ref{C}) remarking that
\begin{align*}
&L_{alg}(l+k+n-2i+1,l+k+n-2i;f) \\
&=-\frac{\Gamma_{\mathbb C}(l+k+n-2i+1)\Gamma_{\mathbb C}(l+k+n-2i)L(l+k+n-2i+1,f)L(l+k+n-2i,f)}{\langle f, \ f \rangle}\end{align*}
for any $1 \le i \le n$.
\end{proof}
\begin{rem} Theorem \ref{thm.standard-L-of-IM-lift1} has been proved in  \cite[Corollary to Theorem 3.1]{I-K-P-Y} assuming Conjecture \ref{conj.A'}. We also note that the definition of $L_{alg}(l+k+n-2i+1,l+k+n-2i;f)$ is a little different from that in \cite{I-K-P-Y}.
\end{rem}

Rewriting Theorem \ref{thm.standard-L-of-IM-lift1}, we have the following result.

\begin{thm}\label{thm.standard-L-of-IM-lift2} Let the notation be as in Theorem \ref{thm.standard-L-of-IM-lift1}. Then
\begin{align*}
&|c_{\widetilde {\mathcal F}_{h,g}}(A)|^2 L_{alg}(l, \widetilde {\mathcal F}_{h,g} , \mathrm {St}) \\
&=(-1)^{n+[(n+1)/2]}2^{k-1}|D|^k C_{h,g}{\dfrac{ | c_{{\widetilde {\mathcal F}}_{h,g}}(A)|^2}{|c_h(|D|)|^2}}L_{alg}(k+n,k;f;1,\chi_D) \\
\times  &  { \dfrac{ L_{alg}(l,g ,\mathrm {St}) \prod_{i=1}^{n} L_{alg}(l+k+n-2i+1,l+k+n-2i;f)}
{ L_{alg}(2k+2n, g \otimes g \otimes f)\prod_{i=1}^{n-1} L_{alg}(2i+1,f,\mathrm {St}) \prod_{i=1}^n \widetilde \xi(2i)}}
\end{align*}
 for any positive definite half-integral  matrix  $A$ of degree $2n+1$, a fundamental discriminant $D$ such that $c_h(|D|) \not=0$ and  
an odd integer~$l$ with  $1 \le l \le k-n$. In particular, if $c_h(1) \not=0$, then 
\begin{align}\label{PIM}
&|c_{\widetilde {\mathcal F}_{h,g}}(A)|^2 L_{alg}(l, \widetilde {\mathcal F}_{h,g} , \mathrm {St}) \\
&=(-1)^{n+[(n+1)/2]}2^{k-1}C_{h,g}{\dfrac{ | c_{{\widetilde {\mathcal F}}_{h,g}}(A)|^2}{|c_h(1)|^2}}L_{alg}(k+n,k;f) \notag \\
\times  &  { \dfrac{ L_{alg}(l,g ,\mathrm {St}) \prod_{i=1}^{n} L_{alg}(l+k+n-2i+1,l+k+n-2i;f)}
{ L_{alg}(2k+2n, g \otimes g \otimes f)\prod_{i=1}^{n-1} L_{alg}(2i+1,f,\mathrm {St}) \prod_{i=1}^n \widetilde \xi(2i)}}. \notag
\end{align}
\end{thm}
\begin{proof}
By \cite{K-Z}, we have 
\[\dfrac{|c_h(|D|)|^2}{\langle h, \ h \rangle}=\frac{2^{k-1} |D|^{k-1}\Gamma_{\mathbb C}(k)L(k,f,\chi_D)}{|D|^{1/2}\langle f, \ f \rangle}.\]
Moreover we have
\begin{align*}
&L_{alg}(k+n,k;f,1;\chi_D)=\frac{\Gamma_{\mathbb C}(k+n)\Gamma_{\mathbb C}(k)L(k+n,f)L(f,k,\chi_D)}{\sqrt{-1}^{2k+n+1}s(D)|D|^{1/2}\langle f, \ f \rangle},
\end{align*}
where $s(D)=\begin{cases} 1 & \text{ if } k \text{ is even} \\
\sqrt{-1} & \text{ if } k \text{ is odd}. \end{cases}$.
Thus the assertion follows from Theorem \ref{thm.standard-L-of-IM-lift1}.
\end{proof}

\begin{thm} \label{thm.algebraicity-of-period1} 
Let the notation be as in Theorem \ref{thm.standard-L-of-IM-lift2}. Then,
$C_{h,g}$ belongs to $\mathbb Q(f) \mathbb Q(g)$.
\end{thm}
\begin{proof}
We can take a positive definite half-integral  matrix  $A$ of degree $2n+1$, a fundamental discriminant $D$ and  an odd 
integer~$l$ with  $n+1 \le l \le k-n$
 such that $c_{\widetilde {\mathcal F}_{h,g}}(A) \not=0$ and $c_h(|D|) \not=0$. Then, by construction, $\dfrac{|c_{\widetilde {\mathcal F}_{h,g}}(A)|^2} {| c_h(|D|)|^2}$ is non-zero and belongs to $\mathbb Q(f)\mathbb Q(g)$.
Moreover, for any odd  integer $3 \le l \le k-n-1$, the value  $|c_{\widetilde {\mathcal F}_{h,g}}(A)|^2 L_{alg}(l, \widetilde {\mathcal F}_{h,g} , \mathrm {St})$ 
belongs to ${\mathbb Q}(f){\mathbb Q}(g)$. We also note that the value
$L_{alg}(k+n,k;f;1,\chi_D) L_{alg}(l,g ,\mathrm {St}) \prod_{i=1}^{n} L_{alg}(l+k+n-2i+1,l+k+n-2i;f)$ is non-zero and belongs to ${\mathbb Q}(f){\mathbb Q}(g)$.
This proves the assertion.
\end{proof}
\begin{rem} \label{rem.algebraicity-Ikeda-Miyawaki}
As stated before, $L_{alg}(1,g,\mathrm{St})$ is a rational number and $\mathbb Q(\widetilde {\mathcal F}_{h,g})=\mathbb Q(f)\mathbb Q(g)$. Therefore, by Theorems \ref{thm.standard-L-of-IM-lift2} and \ref{thm.algebraicity-of-period1}, $L_{alg}(1,\widetilde {\mathcal F}_{h,g},\mathrm{St})$ is algebraic and belongs to $\mathbb Q(\widetilde {\mathcal F}_{h,g})$.
\end{rem}

We note that
$L_{alg}(k+n,\mathrm {St}(g)\otimes f)$ belongs to $\mathbb Q(f)\mathbb Q(g)$.
Therefore we deduce the following algebraicity for the ratio of the periods.

\begin{thm} \label{thm.algebraicity-of-period2}
The ratio
$\dfrac{\langle  \widetilde {\mathcal F}_{h,g}, \  \widetilde {\mathcal F}_{h,g} \rangle}{\langle f, \ f \rangle^{n-1} \langle h, \ h \rangle \Omega_{-}(f)} $ belongs to $\mathbb Q(f)\mathbb Q(g)$,
and so does
$\dfrac{\langle   {\mathcal F}_{h,g}, \   {\mathcal F}_{h,g} \rangle}{\langle g, \ g \rangle^2 \langle f, \ f \rangle^{n-1} \langle h, \ h \rangle \Omega_{-}(f)}$.
\end{thm}

\section {Pullback of Siegel Eisenstein series}
To prove our main result, we give an algorithm for computing the standard $L$-value for a cuspidal Hecke eigenform for $\Sp_3(\mathbb Z)$ (cf. Theorem \ref{thm.special-pullback}).  This type of algorithm was given in the case $F$ is a cuspidal Hecke eigenform for $\Sp_m(\mathbb Z)$ with $m \le 2$ (cf. \cite{D-I-K}, \cite{I-K}, \cite{Kat2}, \cite{Kat4}).
To give such an algorithm, first we express a certain modular form arising from the pullback of Siegel Eisenstein series  as a linear combination of Hecke eigenforms with their standard $L$-values as coefficients (cf. Theorem \ref{th.pullback-Eisenstein}).
We have carried it out in \cite{Kat22}, and here 
we give its refinement.  
For a non-negative integer $m$, put
\[\Gamma_m(s)=\pi^{m(m-1)/4}\prod_{i=1}^m \Gamma(s-{i-1 \over 2}).\]

 Let $n,l$ be positive integers such that $l$ is even.  Define the Eisenstein series $E_{n,l}(Z,s)$  by
$$E_{n,l}(Z,s)=\bigl(\det \mathrm {Im}(Z)\bigr)^s \sum_{\gamma \in   \Sp_n(\mathbb Z)_{\infty} \backslash \Sp_n(\mathbb Z)}j(\gamma,Z)^{-l}|j(\gamma,Z)|^{-2s},$$
where $\Sp_n(\mathbb Z)_{\infty}=\Big\{\begin{pmatrix} A & B \\ 0_n & D\end{pmatrix} \in \Sp_n(\mathbb Z) \Big\}$.
 Then $E_{n,l}(Z,s)$ converges absolutely as a function of $s$ if the real part of $s$ is large enough.
It has a meromorphic continuation to the whole $s$-plane, and  it   belong to  
$M_l^{\infty}(\Sp_n(\mathbb Z))$. Moreover  it is holomorphic and finite at $s=0$, which will be denoted by $E_{n,l}(Z)$. 
In particular, if $E_{n,l}(Z)$ belongs to $M_l(\Sp_n(\mathbb Z)),$ we put
$$E_{n,l}^*(Z)=Z(n,l)E_{n,l}(Z),$$
where
\[Z(n,l)=\zeta(1-l)\prod_{i=1}^{[n/2]} \zeta(1+2i-2l).\]
From now on, for an element $A \in \mathcal H_n(\mathbb Z)_{\ge 0}$,  we denote by $c_{n.l}(A)$ the $A$-th Fourier coefficient $c_{E_{n,l}^*}(A)$ of $E_{n,l}^*(Z)$.
To determine the Fourier coefficient of $E_{n,l}^*(Z)$, we define a polynomial attached to local Siegel series.
For a prime number $p$, let $\mathbb {Q}_p$ be the field of $p$-adic numbers and $\mathbb Z_p$ the ring of $p$-adic integers. For an element $B \in \mathcal {H}_n(\mathbb {Z}_p)$,
we define the Siegel series $b_p(B,s)$  as 
\[b_p(B,s)=\sum_{R \in \mathrm{Sym}_n(\mathbb {Q}_p)/\mathrm{Sym}_n(\mathbb {Z}_p)} 
{\bf e}_p(\mathrm{tr}(BR))\nu_p(R)^{-s},\]
where ${\bf e}_p$ is the additive character of $\mathbb {Z}_p$ such that ${\bf e}_p(m)={\bf e}(m)$ for $m \in \mathbb {Z}[p^{-1}]$, and
$\nu_p(R)=[R\mathbb {Z}_p^n+\mathbb {Z}_p^n:\mathbb {Z}_p^n]$.
We define $\chi_p(a)$ for $a \in {\mathbb {Q}}^{\times}_p $ as follows:
 $$\chi_p(a):=  
 \left\{\begin{array}{cl} 
  +1            & \text { if } \ {\mathbb {Q}}_p(\sqrt {a})={\mathbb {Q}}_p ,\\
  -1            & \text { if }  \ {\mathbb {Q}}_p(\sqrt {a})/{\mathbb {Q}}_p \text { is quadratic  unramified},\\
  0             & \text { if }  \ {\mathbb {Q}}_p(\sqrt {a})/{\mathbb {Q}}_p \text { is quadratic  ramified}. 
\end{array}\right.$$
For an element  $B \in \mathcal {H}_n(\mathbb {Z}_p)^{\mathrm {nd}}$ with $n$ even,  we define $\xi_p(B)$ by 
$$\xi_p(B):=\chi_p((-1)^{n/2}\det B).$$ 
For a nondegenerate half-integral matrix $B$ of size $n$ over ${\mathbb {Z}}_p$, define a polynomial $\gamma_p(B,X)$ in $X$ by
$$\gamma_p(B,X):=
\left\{
\begin{array}{ll}
(1-X)\prod_{i=1}^{n/2}(1-p^{2i}X^2)(1-p^{n/2}\xi_p(B)X)^{-1} & \text { if }  n \text { is \ even}, \\
(1-X)\prod_{i=1}^{(n-1)/2}(1-p^{2i}X^2) & \text { if } \ n \text { is  odd}.
\end{array}
\right.$$
Then it is well known that  there exists a unique polynomial $F_p(B,X)$ in $X$ over   ${\mathbb {Z}}$  with constant term $1$ such that 
$$b_p(B,s) =\gamma_p(B,p^{-s})F_p(B,p^{-s})$$ (e.g. \cite{Kat99}). 
                             
For  $B \in \mathcal {H}_m(\mathbb {Z})_{>0}$ with $m$ even, 
 let ${\mathfrak d}_B$ be the discriminant of  ${\mathbb {Q}}(\sqrt{(-1)^{m/2}\det B})/{\mathbb {Q}}$,  and $\chi_B=({\frac{{\mathfrak d}_B}{*}})$  the Kronecker character corresponding to ${\mathbb {Q}}(\sqrt{(-1)^{m/2}\det B})/{\mathbb {Q}}$. We note that we have $\chi_B(p)=\xi_p(B)$ for any prime $p.$                                    
 We also note that 
\[(-1)^{m/2}\det (2B)=\mathfrak d_B \mathfrak {f}_B^2\]
with $\mathfrak {f}_B \in \mathbb {Z}_{>0}$.
For any $T \in \mathcal {H}_n(\mathbb {Z}_p)$ which is not-necessarily non-degenerate, we define a polynomial $F_p^*(T,X)$  as follows:
For an element $T \in \mathcal {H}_n(\mathbb {Z}_p)$ of rank $r \ge 1,$ there exists an element 
$\widetilde T \in \mathcal {H}_r(\mathbb {Z}_p)^{\mathrm {nd}}$ such that $T \sim_{\mathbb {Z}_p} \widetilde T \bot O_{n-r}.$ We note that  
$F_p(\widetilde T,X)$ depends only on $T$ and  does not depend on the choice of $\widetilde T.$ Then we put $F_p^\ast(T,X)=F_p(\widetilde T,X).$  
For an element $T \in \mathcal {H}_n(\mathbb {Z})_{\ge 0}$ of rank $r \ge 1,$ there exists an element $\widetilde T \in 
\mathcal {H}_r(\mathbb {Z})_{>0}$ such that $T \sim_{\mathbb {Z}} \widetilde T \bot O_{n-r}.$
If $r$ is even, $\chi_{\widetilde T}$ depend only on $T$ and does not depend on the choice of $\widetilde T$. Then we write $\chi_T^{\ast}=\chi_{\widetilde T}$. 
The following proposition is due to \cite[Proposition 6.3]{A-C-I-K-Y23}.
\begin{prop} \label{prop.FC-Siegel-Eisenstein}
Let $n$ and $l$ be positive integers such that $l \ge n+1$.
 Then $E_{2n,l}^*(Z)$ is holomorphic and belongs to $M_l(\Sp_{2n}(\mathbb Z))$ except in the following case:
$l=n+1 \equiv 2 \text{ mod } 4$. 
In the case that $E_{2n,l}^*(Z)$ is holomorphic we have the following assertion:
For $B \in \mathcal {H}_{2n}(\mathbb {Z})_{\ge 0}$ of rank $m,$ we have
\begin{align*}
c_{2n,l}(B)
&=2^{[(m+1)/2]}\prod_{p \mid  \det (2\widetilde B)} F_p^\ast(B,p^{l-m-1})\\
&\times 
\begin{cases}
\prod_{i=m/2+1}^{n} \zeta(1+2i-2l) L(1+m/2-l,\chi_B^\ast)
 & \text { if }  m \text { is  even  }, \\
\prod_{i=(m+1)/2}^{n} \zeta(1+2i-2l) & \text { if } \ m \text { is  odd}.
\end{cases}
\nonumber
\end{align*}
Here we make the convention that $F_p^*(B,p^{l-m-1})=1$ and $\L(1+m/2-l,\chi^*_B)=\zeta(1-l)$ if $m=0$.
\end{prop}

Let $\stackrel {\!\!\!\!\!\! \circ} {{\mathcal D}_{m,l}^{\nu}}$ be the differential operator in \cite{B-S}, which maps 
$M_{l}(\Sp_{2m}(\mathbb  Z))$ to $M_{l+\nu}(\Sp_m(\mathbb  Z)) \otimes M_{l+\nu}(\Sp_m(\mathbb Z)).$  
 For a non-negative integer $\nu \le k$, we define a function 
${\mathfrak  E}_{2n}^{k,\nu}(Z_1,Z_2)$ on ${\mathbb H}_n \times {\mathbb H}_n$ as
\begin{align*}
&{\mathfrak  G}_{2n}^{k,\nu}(Z_1,Z_2) = (2\pi \sqrt{-1})^{-n\nu}\sqrt{-1}^{-k+\nu} \stackrel {\!\!\!\!\!\!\!\!\!\!\!\!\!\!\! \circ} {{\mathcal D}_{n,k-\nu}^{\nu}}E^*_{2n,k-\nu}
(Z_1,Z_2)
\end{align*}
for $(Z_1,Z_2) \in {\mathbb H}_n \times {\mathbb H}_n$. 
From now on, until Corollary \ref{cor.pullback-Eisenstein}, let $l$ be an even integer such that $l \ge n+1$ and  we assume that  $l$ does not satisfy $l=n+1 \equiv 2 \mod 4$. 
Let $T$ be a $2n \times 2n$ symmetric matrix of variables. Then, there exists a polynomial $P_{n,l}^{\nu}(X)$ in $T$ such that
\begin{align*} &\stackrel {\!\!\!\!\!\! \circ} {{\mathcal D}_{n,l}^{\nu}}\Bigl({\bf e}\bigl(\mathrm{tr}\Bigl(\begin{pmatrix} A_1 & R/2 \\ {}^tR/2 & A_2 \end{pmatrix} \begin{pmatrix} Z_1 & Z_{12} \\ {}^t Z_{12} & Z_2 \end{pmatrix}\Bigr)\Bigr)\Bigr)\\
&=(2\pi \sqrt{-1})^{n \nu}P_{n,l}^{\nu}\Bigl(\begin{pmatrix} A_1 & R/2 \\ {}^tR/2 & A_2 \end{pmatrix}\Bigr) 
{\bf e}(\mathrm{tr}(A_1Z_1+A_2Z_2)) 
\end{align*}
for $\begin{pmatrix} A_1 & R/2 \\ {}^tR/2 & A_2 \end{pmatrix} \in \mathcal {H}_{2n}(\mathbb {Z})_{\ge 0}$ with $A_1,A_2 \in \mathcal {H}_n(\mathbb {Z})_{\ge 0}$
and $ \begin{pmatrix} Z_1 & Z_{12} \\ {}^t Z_{12} & Z_2 \end{pmatrix} \in {\mathbb H}_{2n}$ with $Z_1,Z_2 \in {\mathbb H}_n$.
For $A_1,A_2 \in \mathcal H_n(\mathbb Z)_{\ge 0}$, put
\begin{align}\label{FC}
&c_{{\mathfrak  G}_{2n}^{k,k-l}}(A_1,A_2)\\
&=
\sum_{R \in M_n(\mathbb {Z})} P_{n,l}^{k-l}\Bigl(\begin{pmatrix} A_1 & R/2 \\ {}^tR/2 & A_2 \end{pmatrix}\Bigr)c_{2n,l}\Bigl(\begin{pmatrix} A_1 & R/2 \\ {}^tR/2 & A_2 \end{pmatrix}\Bigr). \notag
\end{align}
Then, by  definition, 
${\mathfrak  G}_{2n}^{k,k-l}(Z_1,Z_2)$ can be expressed as 
\begin{align*}
{\mathfrak  G}_{2n}^{k,k-l}(Z_1,Z_2) =\sum_{A_1,A_2 \in \mathcal {H}_n(\mathbb {Z})_{\ge 0}} c_{{\mathfrak  E}_{2n}^{k,k-l}}(A_1,A_2){\bf e}\Big(\mathrm{tr}(A_1Z_1+A_2Z_2)\Big).
\end{align*}
For $A \in \mathcal H_n(\mathbb Z)_{\ge 0}$, define a Fourier series ${\mathcal G}_{2n}^{k,k-l}(Z_1)$ on $\mathbb H_n$ as
\[{\mathcal G}_{2n}^{k,k-l}(Z_1,A)=\sum_{A_1 \in \mathcal H_n(\mathbb Z)_{\ge 0}} c_{{\mathfrak  E}_{2n}^{k,k-l}}(A_1,A){\bf e}(A_1Z_1).\]
Then, by definition, 
${\mathfrak  G}_{2n}^{k,k-l}(Z_1,Z_2)$ can be expressed as
$${\mathfrak  G}_{2n}^{k,k-l}(Z_1,Z_2)=\sum_{A \in {\mathcal H}_n(\mathbb {Z})_{\ge 0}}{\mathcal G}_{2n}^{k,k-l}(Z_1,A){\bf e}(\mathrm {tr}(AZ_2)).$$
The Fourier series ${\mathcal G}_{2n}^{k,k-l}(Z_1,A)$  belongs to $M_k(\Sp_n(\mathbb Z))$, and in particular, if $l<k$, then ${\mathcal G}_{2n}^{k,k-l}(Z_1,A)$ belongs to 
$S_k(\Sp_n(\mathbb Z))$ (cf. \cite[Proposition 3,6]{Kat22}.
Define $\Lambda(m,F,\mathrm{St})$ by
\[\Lambda(m,F, \mathrm{St})=\frac{
\Gamma(m)\prod_{i=1}^n \Gamma(2k-n-i)L(m,F, \mathrm {St})} { \pi^{-n(n+1)/2+nk+(n+1)m}}.\]
The following lemma is a modification of  \cite[Theorem 3.1]{B-S}.
\begin{lem} \label{lem.pullback}
Let $g$ be a Hecke eigenform in $S_k(\Sp_n(\mathbb Z))$. Then we have
\begin{align*}
\langle g, \mathfrak E_{2n}^{k,k-l}(*,-\overline{Z_2}) \rangle=c(n,k,l)\Lambda(l-n,g,\mathrm {St})g(Z_2),
\end{align*}
where   $c(n,k,l)=(-1)^{l(n-1)/2-n(n+1)/2}2^{-4kn+(n-1)l+2n^2+3n+2}$.
\end{lem}
\begin{proof}
By \cite[Theorem 3.1]{B-S}, we have
\begin{align*}
&\langle g, \  \mathfrak E_{2n}^{k,k-l}(*,-\overline{Z_2}) \rangle\\
&=\langle g, \  (2\pi \sqrt{-1})^{n(l-k)} Z(2n,l)\stackrel {\!\!\!\!\!\! \circ} {{\mathcal D}_{n,l}^{k-l}}(E_{2n,l})(*,-\overline{Z_2}) \rangle \\
&=(2\pi)^{n(l-k)}(-1)^{n(l-k)/2} Z(2n,l) \\
&\times  (-1)^{kl/2}2^{1-nk+n(n+3)/2}\pi^{n(n+1)/2}\dfrac{\widetilde \Omega_{k,l}}{\widetilde Z(2n,l)} L(l-n,g,\mathrm{St})g(Z_2),\end{align*}
where
\[\widetilde Z(2n,l)=\zeta(l)\prod_{i=1}^n \zeta(2l-2i)\]
and
\[\widetilde \Omega_{k,l}=\dfrac{\prod_{i=1}^n \Gamma(k-n/2-(i-1)/2)\prod_{i=1}^n \Gamma(k-(n+1)/2-(i-1)/2)}
{\prod_{i=1}^n \Gamma(l-(i-1)/2)\prod_{i=1}^n \Gamma(l-n/2-(i-1)/2)}.\]
By the functional equation of  the Riemann zeta function, we can check that
\begin{align*}
\dfrac{Z(2n,l)}{\widetilde Z(2n,l)}&=(-1)^{-l/2-n(n+1)/2} \pi^{-l-2ln+n(n+1)}2^{-l-2ln+n(n+1)+n+1}\\
&\times \Gamma(l)\prod_{i=1}^n \Gamma(2l-2i).
\end{align*}
Moreover, by the duplication formula for the Gamma function, we show that 
\begin{align*}
&\prod_{i=1}^n \Gamma(k-n/2-(i-1)/2)\prod_{i=1}^n \Gamma(k-(n+1)/2-(i-1)/2)\\
&=\prod_{i=1}^n2^{-2k+n+i+1}\pi^{1/2}\Gamma(2k-n-i),\end{align*}
and
\begin{align*}
&\prod_{i=1}^n \Gamma(l-(i-1)/2)\prod_{i=1}^n \Gamma(l-n/2-(i-1)/2)\\
& \Gamma(l) \Gamma(l-n)^{-1}\prod_{i=1}^n (\Gamma(l-(2i-1)/2)\Gamma(l-2i/2)) \\
&=\Gamma(l)\Gamma(n-l)^{-1}\prod_{i=1}^n 2^{-2l+2i+1} \pi^{1/2}\Gamma(2l-2i).\end{align*} 
This proves the assertion.

\end{proof}

The following theorem is essentially a refinement of a special case of \cite[Theorem 3.7]{Kat22}.
\begin{thm} \label{th.pullback-Eisenstein}
Let $\{F_i \}_{i=1}^d$ be an orthogonal basis of $ S_k(\Sp_n(\mathbb Z))$ consisting of Hecke eigenforms, and $\{F_i \}_{d+1 \le i \le e}$ be a basis of the orthogonal complement $S_k(\Sp_n(\mathbb Z))^\perp$ of $S_k(\Sp_n(\mathbb Z))$ in $M_k(\Sp_n(\mathbb Z))$ with respect to the Petersson product.
Then, for any $A \in \mathcal H_n(\mathbb Z)_{>0}$ we have
\begin{align*}
{\mathcal G}_{2n}^{k,k-l}(Z,A) &= c(n,k,l)\sum_{i=1}^d\frac{\Lambda(l-n,F_i, \mathrm {St})}{\langle F_i, F_i \rangle} \overline{c_{F_i}(A)}F_i(Z)\\
&+ \sum_{i=d+1}^e c_{i,A} F_i(Z),
\end{align*}
where $c_{i,A}$ is a certain complex number. Moreover we have $c_{i,A}=0$ for any $d+1 \le i \le e$ if $l<k$.
\end{thm}
\begin{cor}\label{cor.pullback-Eisenstein}
Put 
$$ \widetilde {\mathcal G}_{2n}^{k,k-l}(Z,A) =\widetilde c(n,k,l){\prod_{i=1}^n  \Gamma(l-n+k-i) \over 
\prod_{i=1}^n \Gamma(2k-n-i) }
{\mathcal G}_{2n}^{k,k-l}(Z,A),$$
where $\widetilde c(n,k,l)=(-1)^{l(n-1)/2+n(n+1)/2}2^{3kn-2nl-n(n+1)/2-1}$.  
Then we have
\begin{align*}
\widetilde {\mathcal G}_{2n}^{k,k-l}(Z,A) &=\sum_{i=1}^d L_{alg}(l-n,F_i,  \mathrm {St}) \overline{c_{F_i}(A)}F_i(Z)\\
&+ \sum_{i=d+1}^e \widetilde c_i F_i(Z),
\end{align*}
where  $\widetilde c_{i,A}$ is a certain complex number. Moreover we have $\widetilde c_{i,A}=0$ for any $d+1 \le i \le e$ if $l<k$.
\end{cor}
\begin{proof}
For any Hecke eigenform $F$ in $S_k(Sp_n(\mathbb Z))$ we have 
\begin{align*}
&L_{alg}(l-n,F,\mathrm{St})=2^{e(k,l,n)}{\prod_{i=1}^n \Gamma(l-n+k-i) \over 
\prod_{i=1}^n \Gamma(2k-n-i)}\frac{\Lambda(l-n,F,\mathrm{St})}{\langle F, F \rangle},
\end{align*}
where  $e(k,l,n)=-kn-(n+1)l+3n^2/2+5n/2+1$.
This proves the assertion.
\end{proof}

Now we give an explicit formula for $\stackrel {\!\!\!\!\!\! \circ} {{\mathcal D}_{n,k}^{2}}$.
The formula has been known in \cite[p. 289 (the top of (4.2.1))]{I-K-O} for any $n$ and $k$, and here we review it. Let
\[
T=\begin{pmatrix} R & W \\ ^{t}W & S \end{pmatrix}
\]
be a $2n \times 2n$ symmetric matrix of variables, where
$R,W,S$ are $n \times n$ matrices. We define $P_{\alpha}(T)$ for $0\leq \alpha \leq n$ by 
\[
\det\Big(\begin{pmatrix} xR & W \\ ^{t}W & S \end{pmatrix}\big)
=\sum_{\alpha=0}^{n}x^{\alpha}P_{\alpha}(T).\]
Define a polynomial  $Q_{n,k}^2(T)$  in $T$ as 
\[Q_{n,k}^2(T)=
\sum_{\gamma=0}^{n}\binom{n}{\gamma}^{-1}\binom{2k-n+1}{n-\gamma}P_{\gamma}(T),
\]
and put 
\[\mathbb D_{n,k}^2=Q_{n,k}^2(\partial _Z),\]
where $\partial_Z=\begin{pmatrix} {1+\delta_{ij} \over 2} {\partial 
\over \partial z_{ij}} \end{pmatrix}_{1\le i,j \le n}.$
We denote by $\mathrm{Hol}(\mathbb H_m)$ the $\mathbb C$-vector space of  holomorphic functions on $\mathbb H_m$. We then define
a mapping $\widetilde {\mathcal D}_{n,k}^2$ from $\mathrm{Hol}(\mathbb H_{2n})$ to
$\mathrm{Hol}(\mathbb H_n) \otimes \mathrm{Hol}(\mathbb H_n)$ as 
$$\widetilde {\mathcal D}_{n,k}^2(F)=\mathrm{Res}(\mathbb  {D}_{n,k}^2(F))$$
for $F \in \mathrm{Hol}(\mathbb H_{2n})$, 
where $\mathrm{Res}$ is the restriction of a function on 
$\mathbb H_{2n}$ to $\mathbb H_n \times \mathbb H_n$.
\begin{prop} \label{prop.dfiff-op}
 The operator 
$\widetilde {\mathcal D}_{n,k}^2$  maps an element $F$ of
$S_k(\Sp_{2n}(\mathbb Z))$ to $S_{k+2}(\Sp_n(\mathbb Z)) \otimes S_{k+2}(\Sp_n(\mathbb Z))$.
\end{prop}
\begin{proof}
The assertion is written in  \cite[4.2.1 in p. 289]{I-K-O}, and easily proved by a simple computation using \cite[Proposition 4.1]{I-K-O}.
\end{proof}
We consider the special case $n=3$. Then we can derive the following formula:
\begin{align*}
P_0(T)&=-(\det W)^2,\\
P_1(T)&=\sum_{i_1=1}^3 \sum_{4<i_2<i_3 \le 6, 1 \le i_4<i_5<i_6 \le 6 \atop
\{i_4,i_5,i_6 \} \cap \{i_1,i_2,i_3 \}=\emptyset} (-1)^{i_1+i_2+i_3} \det \Big(T \begin{pmatrix} i_1 & i_2 & i_3 \\ 1 & 2 &3 \end{pmatrix}\Big)  \det \Big(T \begin{pmatrix} i_4 & i_5 & i_6 \\ 4 & 5 &6 \end{pmatrix}\Big),\\
P_3(T)&=\det R \det S,\\
& \text{ and } \\
P_2(T)&=\det T-P_0(T)-P_1(T)-P_3(T),\end{align*}
and
\[Q_{3,k}^2(T)=\dfrac{2(k-1)(2k-3)(k-2)}{3}P_0(T)+\dfrac{(k-1)(2k-3)}{3}P_1(T)+\dfrac{2(k-1)}{3}P_2(T)+P_3(T).\]
To calculate the critical values of $L(s,F,\mathrm{St})$ for a Hecke eigenform $F$ in $S_{k+2}(\mathrm{Sp}_3(\mathbb {Z}))$, we have to compare the differential operator $\widetilde {\mathcal D}_{3,k}^2$ with $\stackrel {\!\!\!\!\!\! \circ} {{\mathcal D}_{3,k}^{2}}$.
\begin{prop}\label{prop.comparison-differential-operator}
We have 
$$\stackrel {\!\!\!\!\!\! \circ} {{\mathcal D}_{3,k}^{2}}=-2^{-4}3(2k-1)(2k-3)(k-1)\widetilde {\mathcal D}_{3,k}^2.$$
\end{prop}
\begin{proof}
We note that 
 $\widetilde {\mathcal D}_{3,k}^{2}$ is  a 
constant multiple of $\stackrel {\!\!\!\!\!\! \circ} {{\mathcal D}_{3,k}^{2}}$  as 
a mapping from $\mathrm{Hol}(\mathbb H_6)$ to $\mathrm{Hol}(\mathbb H_3) \otimes \mathrm{Hol}(\mathbb H_3)$.
By \cite[(1.21)]{B-S}, we have 
\[\stackrel {\!\!\!\!\!\! \circ} {{\mathcal D}_{3,k}^{2}}((\det W)^2)=\prod_{\nu=1}^2 C_3(\nu/2)C_3(k-1-\nu/2).\]
On the other hand, we have 
\begin{align*}
&\widetilde  {\mathcal D}_{3,k}^2((\det W)^2)\\
&=-\dfrac{2(k-1)(2k-3)(k-2)}{3}P_0(\partial_Z)((\det W)^2)\\
& =-\dfrac{2(k-1)(2k-3)(k-2)}{3}C_3(1)C_3(1/2),
\end{align*}
where $C_3(s)=s(s+1/2)(s+1)$. This proves the  assertion.
\end{proof}
\begin{prop}\label{prop.FC}
For $A \in \mathcal H_3(\mathbb Z)_{>0}$, let 
$\widetilde {\mathcal G}_{6}^{k+2,2}(Z,A)$ be the element of $S_{k+2}(\Sp_3(\mathbb Z))$ defined in Corollary \ref{cor.pullback-Eisenstein}. Then, for any $A \in \mathcal H_3(\mathbb Z)_{>0}$, we have 
\begin{align}
&c_{\widetilde {\mathcal G}_{6}^{k+2,k}(*,A)}(A_1)={-3 \Gamma(2k-3) \Gamma(2k-4) \over 
 \Gamma(2k)\Gamma(2k-1) }2^{3k+7}(2k-1)(2k-3)(k-1)\\
& \times \sum_{R \in M_3(\mathbb Z)}
c_{6,k}\Big(\begin{pmatrix} A_i & R/2 \\ {}^tR/2 & A \end{pmatrix} \Big)Q_{3,k}^2\Big(\begin{pmatrix} A_i & R/2 \\ {}^tR/2 & A \end{pmatrix} \Big).\notag
\end{align}
\end{prop}

The following theorem is a key to proving our main results.
\begin{thm}\label{thm.special-pullback} Let $F$ be a Hecke eigenform in 
$S_{k+2}(\Sp_3(\mathbb Z))$. 
Let $\dim S_{k+2}(\Sp_3(\mathbb Z))=d$ and that $F_1,\ldots,F_d$ be a basis of $S_{k+2}(\Sp_3(\mathbb Z))$ consisting of Hecke eigenforms such that $F_1=F$. Let $A_1,\ldots, A_d$ and $A$ be elements of $\mathcal H_3(\mathbb Z)_{>0}$ and put $a_{ij}=c_{F_j}(A_i)$ for $i=1,\ldots,d$ and $j=1,\ldots,d$.
Moreover, put $C(k;A_i,A)=c_{\widetilde {\mathcal G}_{6}^{k+2,k}(*,A)}(A_1)$.
Here $c_{6,k}(*)$ is the Fourier coefficient of the Siegel Eisenstein series $E_{6,k}^*$ as stated at the beginning of this section.
Suppose that $\det (a_{ij})_{1 \le i,j \le d} \not=0$.
Then,  we have
\begin{align*}|c_{F}(A)|^2 L_{alg}(k-3, F, \mathrm {St})&=
\frac{c_F(A)\begin{vmatrix} C(k;A_1,A) & a_{1,2} & \hdots & a_{1,d} \\
\vdots & \vdots & \ddots & \vdots \\
C(k;A_d,A) & a_{d,2} & \hdots & a_{d,d} \end{vmatrix}}{\det (a_{ij})_{1 \le i,j \le d}}.
\end{align*} 
In particular if $\dim S_{k+2}(\Sp_3(\mathbb Z))=1$, then
\begin{align*}|c_{F}(A)|^2 L_{alg}(k-3, F, \mathrm {St})=C(k;A,A).
\end{align*}
\end{thm}
\begin{proof}
By Corollary \ref{cor.pullback-Eisenstein} and Proposition \ref{prop.comparison-differential-operator}, for any $1 \le i \le d$, we have
\[c_F(A)C(k;A_i,A)=\sum_{j=1}^d a_{ij} c_F(A)\overline{c_{F_j}(A)}L_{alg}(k-3,F_j,\mathrm{St}).\]
Thus the assertion can be derived by  Cramer's formula.
\end{proof}
\begin{rem}\label{rem.special-pullback} In view of Theorem \ref{thm.standard-L-of-IM-lift2}, to confirm Conjecture \ref{conj.A'} for the Ikeda-Miyawaki lift $F$ in $S_{k+2}(\Sp_3(\mathbb Z))$, 
we  have to compute the $L$-value $L(m,F,\mathrm{St})$ at least one  critical point $m$ such that $1 \le m \le k-1$ and $m \equiv 1 \mod 2$.
Besides $m=k-3$, we can obtain a formula  also for $|c_{F}(A)|^2 L_{alg}(k-1, F, \mathrm {St})$ using Corollary \ref{cor.pullback-Eisenstein}. In this case $\widetilde {\mathcal G}_{6}^{k+2,0}(Z,A)$ is essentially obtained by  the restriction of
$E_{6,k+2}^*$ without differential operator.
However, to get it we have to solve a system of linear equations of $e$-variables with $e=\dim_{\mathbb C} M_{k+2}(Sp_3(\mathbb Z))$.
 On the other hand, in the above proposition, we only have to solve  a system of linear equations of $d$-variables with $d=\dim S_{k+2}(\Sp_3(\mathbb Z))$. 
For example, in the case $k=14$, we have $d=7$ and $e=3$.
This implies that the pullback formula with differential operator allows us to compute the value in question efficiently.
\end{rem}
We have an explicit formula for $F_p^*(T,X)$ for any half-integral matrix $T$ over ${\mathbb Z}_p$ (cf. \cite{Kat99}, \cite{Ikeda-Katsurada18},\cite{Ikeda-Katsurada22}), and an algorithm for computing it (cf. \cite{Lee18}). Therefore, by Proposition \ref{prop.FC-Siegel-Eisenstein}, we can compute 
$c_{6,k}(B)$ for any $B \in \mathcal H_6(\mathbb Z)_{\ge 0}$.
Therefore, by using  the above theorem together with Proposition \ref{prop.FC}, we can compute 
$|c_{F}(A)|^2 L_{alg}(k-3, F, \mathrm {St})$ with the aid of Mathematica, as will be shown in Propositions \ref{prop.numerical-example2} and \ref{prop.numerical-example-16-2}.

\section{Computation of $L_{alg}(l_1,l_2;f)$}
For our later purpose, we explain how to compute $L_{alg}(l_1,l_2;f)$ following \cite{Za}.
Let $l_1,l_2,k$ be positive integers satisfying the following conditions: 
\begin{align}  \label{*} l_1-l_2 \ge  3 \text{ is odd},   k \text{ is even and  }l_1+1 < k \le l_1+l_2-3.
\end{align}
Put $\nu=k-l_1-1$, and define a function $G_\nu(E^*_{l_1-l_2+1},E^*_{l_1+l_2-k+1})(\tau)$ on $\mathbb H_1$ as 
\begin{align*}
&G_\nu(E^*_{l_1-l_2+1},E^*_{l_1+l_2-k+1})(\tau)\\
&=(2\pi \sqrt{-1})^{-\nu}\sum_{\mu=0}^{\nu} (-1)^{\nu-\mu}\begin{pmatrix} \nu \\ \mu \end{pmatrix} \frac{\Gamma(l_1-l_2+1+\nu)\Gamma(l_1+l_2-k+1+\nu)}{\Gamma(l_1-l_2+1+\mu)\Gamma(l_1+l_2-k+1+\nu-\mu)} \\
& \times \frac{\partial^{\mu} E^*_{l_1-l_2+1}}{\partial \tau^{\mu}} \frac{\partial^{\nu-\mu} E^*_{l_1+l_2-k+1}}{\partial \tau^{\nu-\mu}}.
\end{align*}
Here, for a positive even integer $l$, let $E_l^*=E_{1,l}^*$ is the Eisenstein series defined in Section 4.
Then, $G_\nu(E^*_{l_1-l_2+1},E^*_{l_1+l_2-k+1})$ belongs to $S_k(SL_2(\mathbb Z))$ (cf. \cite[p. 147]{Za}).
Let $f_1,\ldots,f_d$ be a basis of $S_k(SL_2(\mathbb Z))$ consisting of primitive forms. Then, by \cite[(77)]{Za}, similarly to Corollary \ref{cor.pullback-Eisenstein} we have
\begin{align*}
G_\nu(E^*_{l_1-l_2+1},E^*_{l_1+l_2-k+1})(\tau)=\gamma(k,l_1) \sum_{i=1}^d 
L_{alg}(l_1,l_2;f) f_i(\tau),\end{align*}
where $\gamma(k,l_1)=\frac{(-1)^{k/2}\Gamma(k-1)}{2^{k-1}\Gamma(l_1)}.$ Thus, similarly to Theorem \ref{thm.special-pullback}, we obtain the following formula for $L(l_1,l_2,f)$
\begin{prop}\label{prop.product-Hecke-L}
Let $l_1,l_2,k$ be positive integers satisfying the conditions (\ref{*}), and $f_1,\ldots,f_d$ be as above. Let $m_1,\ldots,m_d$ be positive integers, and put $a_{ij}=c_{f_j}(m_i)$. Write
\[G_\nu(E^*_{l_1-l_2+1},E^*_{l_1+l_2-k+1})(\tau)=\sum_{m=1}^\infty C(k,l_1,l_2;m){\bf e}(m\tau).\]
Suppose that $\det (a_{ij})_{1 \le i,j \le d} \not =0$ and put $f=f_1$. Then we have
\begin{align*} \gamma(k,l_1) L_{alg}(l_1,l_2;f)&=
\frac{\begin{vmatrix} C(k,l_1,l_2;m_1) & a_{1,2} & \hdots & a_{1,d} \\
\vdots & \vdots & \ddots & \vdots \\
C(k,l_1,l_2;m_d) & a_{d,2} & \hdots & a_{d,d} \end{vmatrix}}{\det (a_{ij})_{1 \le i,j \le d}}.
\end{align*}
\end{prop}
\begin{rem} \label{rem.product-Hecke-L} \begin{itemize}
\item [(1)] A result similar to above holds also for $k=l_1+1$ (i.e. $\nu=0$). 
\item[(2)] The above method  cannot be applied to the case $l_1=l_2+1$ as it is  because $E_{1,2}^*$ is not a holomorphic modular form. But it is nearly holomorphic, and a result similar to above can be obtained by modifying the above method (cf. \cite{Sh0}).
\end{itemize}
\end{rem}

\section{Examples}
To give examples, for an even integer $l \ge 6$, we define a function $\delta_l(\tau)$ on $\mathbb H_1$ as
\[\delta_l(\tau)=\frac{1}{8\pi \sqrt{-1}} \Big((\frac{l}{2}-1)E_{l-2}^*(4\tau)\frac{d}{d\tau}\theta(\tau)-\theta(\tau)\frac{d}{d\tau}E_{l-2}^*(4\tau) \Big),\]
where 
\[\theta(\tau)=1+2\sum_{n=1}^{\infty} {\bf e}(n^2\tau).\]
Then, $\delta_l(\tau)$ belongs to $S_{l+1/2}^+(\varGamma_0(4))$ (cf. \cite[p. 187]{K-Z}).
\subsection{The case $k=10$} Let $k=10$ and $n=1$. 
Then $S_{21/2}^+(\varGamma_0(4))$ and $S_{20}(\SL_2(\mathbb Z))$ are one dimensional, and 
$h(\tau):=120 \delta_6(\tau)$ is a Hecke eigenform in $S_{21/2}^+(\varGamma_0(4))$ with the following  Fourier expansion:
$$h(\tau)=q-56q^4+360q^5-13680q^8 \cdots,$$
where $q={\bf e}(\tau)$.
(cf. \cite[p. 178]{K-Z}). 
The primitive form $f$ in $S_{20}(\SL_2(\mathbb Z))$ corresponding to $h$ has the following Fourier expansion:
$$f(\tau)=q+456q^2+\cdots.$$
Moreover, $S_{12}(\SL_2(\mathbb Z))$ is also one dimensional, and it is spanned by the Ramanujan delta function, which is denoted by $g_{12}$.
Moreover, $S_{12}(\Sp_3({\mathbb Z}))$ is one dimensional and spanned by $\widetilde {\mathcal F}_{h,g_{12}}$. To verify Ikeda's conjecture in this case, we provide several propositions.

\begin{prop}\label{prop.example-Fourier-IM-lift}
Let $h,f$ and $g_{12}$ be as above.
Let $A=\begin{pmatrix} 1 & 0 & 1/2 \\ 0 & 1 & 1/2 \\ 1/2 & 1/2 & 1\end{pmatrix}$.
Then we have \\
$c_{\widetilde {\mathcal F}_{h,g_{12}}}(A)=-2^7 \cdot 3^3 \cdot 5$.
\end{prop}
\begin{proof}
Since $S_{12}(\SL_2(\mathbb Z))$ is one dimensional, we have 
$$I_4(h)\Big(\begin{pmatrix} z & O \\ O & w \end{pmatrix}\Big)=\widetilde F_{h,g_{12}}(z)g_{12}(w).$$
Hence we have 
$$\sum_{r \in M_{3,1}(\mathbb Z)} c_{I_4(h)}\Big(\begin{pmatrix} 
A & r/2 \\ {}^tr/2 & 1 \end{pmatrix}\Big)=c_{\widetilde {\mathcal F}_{h,g_{12}}}(A).$$
By a simple computation, we have
\begin{align*}
&\sum_{r \in M_{3,1}(\mathbb Z)} c_{I_4(h)}\Big(\begin{pmatrix} 
A & r/2 \\ {}^tr/2 & 1 &\end{pmatrix}\Big)\\
&=c_{I_4(h)}\Big(\begin{pmatrix}A & \bf 0 \\
{}^t \bf 0 & 1 \end{pmatrix}\Big)+\sum_{r \in \mathcal R_1} c_{I_4(h)}\Big(\begin{pmatrix} 
A & r/2 \\ {}^tr/2 & 1 \end{pmatrix}\Big)+\sum_{r \in \mathcal R_2} c_{I_4(h)}\Big(\begin{pmatrix} 
A & r/2 \\ {}^tr/2 & 1 \end{pmatrix}\Big),
\end{align*}
where 
$$\mathcal R_1=\Big\{\pm \begin{pmatrix} 0 \\ 1 \\ 0 \end{pmatrix}, \pm \begin{pmatrix} 0 \\ 1 \\ 1 \end{pmatrix},\pm \begin{pmatrix} 1 \\ 0 \\ 0 \end{pmatrix},\pm\begin{pmatrix} 1 \\ 0 \\ 1 \end{pmatrix}\Big\}, \text{ and } \mathcal R_2=\Big\{\pm \begin{pmatrix} 0 \\ 0 \\ 1 \end{pmatrix}, \pm \begin{pmatrix} 1 \\ 1 \\ 1 \end{pmatrix},\pm \begin{pmatrix} 1 \\ -1 \\ 0 \end{pmatrix}\Big\}.$$
\end{proof}
We see that we have $\det \Big(2 \begin{pmatrix}A & \bf 0 \\
{}^t \bf 0 & 1 \end{pmatrix}\Big)=\mathfrak d_{\big(\begin{smallmatrix}A & \bf 0 \\
{}^t \bf 0 & 1 \end{smallmatrix}\big)} =8$. Hence we have 
$$c_{I_4(h)}\Big(\begin{pmatrix}A & \bf 0 \\
{}^t \bf 0 & 1 \end{pmatrix}\Big)=c_h(8)=-13680.$$
Similarly, for any $r \in \mathcal R_1$, we have
$\det \Big(2 \begin{pmatrix}A & r/2 \\
{}^t r/2 & 1 \end{pmatrix}\Big)=\mathfrak d_{\big(\begin{smallmatrix}A & r/2 \\
{}^tr/2 & 1 \end{smallmatrix}\big)} =5$, and 
$$c_{I_4(h)}\Big(\begin{pmatrix}A & r/2 \\
{}^t r/2 & 1 \end{pmatrix}\Big)=c_h(5)=360.$$
For any $r \in \mathcal R_2$, we have $\mathfrak d_{\big(\begin{smallmatrix}A &  r/2 \\
{}^tr/2 & 1 \end{smallmatrix}\big)} =1$ and $\mathfrak f_{\big(\begin{smallmatrix}A & r/2 \\
{}^tr/2 & 1 \end{smallmatrix}\big)} =2$. By \cite{Kat99} we have
\[F_2(\begin{pmatrix}A & r/2 \\
{}^t r/2 & 1 \end{pmatrix},X)=1-12X+32X^2.\]
Hence we have
\begin{align*}
&c_{I_4(h)}\Big(\begin{pmatrix}A & r/2 \\
{}^t r/2 & 1 \end{pmatrix}\Big)=c_h(1)2^{19/2}\alpha_{f,2}^{-1} F_2(\begin{pmatrix}A & r/2 \\
{}^t r/2 & 1 \end{pmatrix},2^{-5/2}\alpha_{f,2})\\
&=2^{19/2}(\alpha_{f,2}+\alpha_{f,2}^{-1}-2^{-1/2}\cdot 3)
=c_f(2)-2^9 \cdot 3=-1080.
\end{align*}
This proves the assertion.

\begin{prop}\label{prop.example-L-values}
\begin{itemize}
\item[(1)] We have $L_{alg}(22,f \otimes g_{12} \otimes g_{12})=\dfrac{2^{51} \cdot 5}{11 \cdot 17}$.
\item[(2)] We have $L_{alg}(7,g,\mathrm{St})=\dfrac{2^{15}}{3}$.
\item[(3)] We have $L_{alg}(11,10;f)L_{alg}(17,16;f)=\dfrac{2^{34} \cdot 13} {3^4\cdot  5 \cdot 17^{2}}$.
\end{itemize}
\end{prop}
\begin{proof}
The assertions (1) and (2) follow from \cite[Table 2]{I-K-P-Y} and \cite[Table1]{Du}, respectively. As explained in Remark \ref{rem.product-Hecke-L}, we cannot compute 
either  $L_{alg}(11,10;f)$ or $L_{alg}(17,16;f)$ individually 
using Proposition \ref{prop.product-Hecke-L}.
However, we have
$$L_{alg}(11,10;f)L_{alg}(17,16;f)=L_{alg}(17,10;f)L_{alg}(16,11;f).$$
Then the value in (3) can be computed using Proposition \ref{prop.product-Hecke-L}.
\end{proof}

\begin{prop} \label{prop.numerical-example2}
Let $h,g_{12}$ and  $A$ be as  above. Then 
$$|c_{\widetilde {\mathcal F}_{h,g_{12}}}(A)|^2 L_{alg}(7, \widetilde {\mathcal F}_{h,g_{12}} , \mathrm {St})=\dfrac{2^{37} \cdot3^2 \cdot 11 \cdot 13}{ 17}.$$
\end{prop}
\begin{proof}
By Theorem \ref{thm.special-pullback}, we have 
\begin{align*} |c_{\widetilde {\mathcal F}_{h,g_{12}}}(A)|^2 L_{alg}(7, \widetilde {\mathcal F}_{h,g_{12}} , \mathrm {St})&=\frac{-3\Gamma(17)\Gamma(16)}{ \Gamma(20)\Gamma(19)}2^{37}\cdot 19\cdot 17\cdot 9\\
&\times\sum_{R \in M_3(\mathbb Z)}c_{6,10}\Big(\begin{pmatrix} A & R/2 \\ {}^tR/2 & A \end{pmatrix} \Big)Q_{3,10}^2\Big(\begin{pmatrix} A & R/2 \\ {}^tR/2 & A \end{pmatrix} \Big).
\end{align*}
Thus the assertion can be verified  by a computation with Mathematica.
\end{proof}
By Propositions \ref{prop.example-Fourier-IM-lift}, \ref{prop.example-L-values}, \ref{prop.numerical-example2} combined with (\ref{PIM}), we can make sure  the following theorem.

\begin{thm}
\label{thm.main-example}
Let $h$ and $g_{12}$ be as above. Then
$C_{h,g_{12}}=2^{14}$, that is,  Conjecture \ref{conj.A'} holds for this case, and so does Conjecture \ref{conj.A}.
\end{thm}
\subsection{The case $k=14$}

Let $k=14$ and $n=1$. Then $S_{28}(\SL_2(\mathbb Z))$ is two dimensional and  spanned by two primitive forms $f_{\pm}(\tau)$
such that
\[f_{\pm}(\tau)=q+(-4140 \pm 108\lambda)q^2+\cdots.\]
The space $S_{29/2}^+(\varGamma_0(4))$ is also  two dimensional and spanned by  
$\delta_6(\tau)E_8^*(\tau)$ and $\delta_8(\tau)E_6^*(\tau)$.
Then, by a simple computation, we see that we can take 
a basis $\{ h_+,h_-\}$ of $S_{29/2}^+(\varGamma_0(4))$ consisiting of Hecke eigenforms such that 
\begin{align*}h_{\pm}(\tau)=q +(- 12332 \pm 108\lambda )q^4 +( 123360 \mp 
 1080 \lambda)q^5 + (1126824 \mp 10152 \lambda)q^8 +\cdots\end{align*}
with $\lambda= \sqrt{18209}$.
We note that $f_{\pm}$ corresponds to $h_{\pm}$ under the Shimura correspondence.
Let $g_{16}$ be the unique primitive form in $S_{16}(\SL_2(\mathbb Z))$.
Let $A$ be as in Section 6.1, and $A_1=\begin{pmatrix} 1 & 0 & 0 \\ 0 & 1 & 1/2 \\ 0 & 1/2 & 1 \end{pmatrix}$, and $A_2=\begin{pmatrix} 1 & 0 & 0 \\ 0 & 1 & 0 \\ 0 & 0 & 1 \end{pmatrix}$. Then, $S_{16}(\Sp_3({\mathbb Z}))$ is three  dimensional and spanned by the Hecke eigenforms $f_1,f_2,f_3$ such that
\[c_{f_1}(A)=c_{f_2}(A)=c_{f_3}(A)=1,\]
\[c_{f_1}(A_1)=3(29+\lambda)/26, c_{f_2}(A_1)=3(29-\lambda)/26,c_{f_3}(A_1)=16,\]
and
\[c_{f_1}(A_2)=2(2293+19\lambda)/13,c_{f_2}(A_2)=2(2293-19\lambda)/13,c_{f_3}(A_2)=-40\]
(cf. \cite[Table 6]{I-K-P-Y} and {\bf Corrections to \cite{I-K-P-Y}} bellow).
We note that 
\[f_1=c\widetilde {\mathcal F}_{h_+,g_{16}} , \  f_2=c\widetilde {\mathcal F}_{h_-,g_{16}} \text{ with some  } c \in \mathbb Q(\lambda)^\times.\]

\begin{prop}
\label{prop.example-Fourier-IM-lift16}
Let $h_+,f_+$ and $g_{16}$ be as above.
Then we have \\
$c_{\widetilde {\mathcal F}_{h_+,g_{16}}}(A)=-2^5 \cdot 567 (-107 + \lambda)$.
\end{prop}
\begin{proof}
Similarly to Proposition \ref{prop.example-Fourier-IM-lift}, we have
\[c_{\widetilde {\mathcal F}_{h_+,g_{16}}}(A)=c_{h_+}(8)+8c_{h_+}(5)+6(c_{f_+}(2)-2^{13}\cdot 3).\]
This proves the assertion.
\end{proof}
\begin{prop}
\label{prop.example-L-values-16}
\begin{itemize}
\item[(1)] We have 
\[L_{alg}(30,f_+ \otimes g_{16} \otimes g_{16})=
-\dfrac{2^{61} \cdot 21 (81594529 + 464593 \lambda)}{2586133225}.
\]
\item[(2)] We have $L_{alg}(11,g_{16},\mathrm{St})=
\dfrac{2^{19} \cdot 3^2 \cdot 839}{13}$.
\item[(3)] We have 
\[L_{alg}(15,14;f_+)L_{alg}(25,24;f_+)=
\dfrac{2^{48} \cdot (26136063 + 188401 \lambda)}{27962195625}.\]
\end{itemize}
\end{prop}
\begin{proof}
(1) On \cite[Table 2]{I-K-P-Y},  only  the norm of $L_{alg}(30,f_+ \otimes g_{16} \otimes g_{16})$ has been given, and  here we give the value itself.
For non-negative integers $i_1,i_2,i_3$, let $b(i_1,i_2,i_3,28,16,16;2)$ be the quantity defined in \cite[p. 166]{I-K-P-Y}, and let 
\[{\bf M}(4,f_+,g_{16},g_{16})=2^{-57}\begin{pmatrix} 28 \\ 14 \end{pmatrix} \begin{pmatrix} 14 \\ 2 \end{pmatrix} L_{alg}(30,f_+ \otimes g_{16} \otimes g_{16})\]
as define in \cite[p.164]{I-K-P-Y}. Let $\Phi(X)$ be the Hecke polynomial of the Hecke operator $T(2)$ acting on $S_{28}(SL_2(\mathbb Z))$. Then we have 
$\Phi(X)=-195250176 + 8280 X + X^2$.  Then by \cite[Theorem 4.9]{I-K-P-Y}, we have
\begin{align*}
&{\bf M}(4,f_+,g_{16},g_{16})\\
&=\Phi'(c_{f_+}(2))^{-1}\sum_{i_1=0}^1 \sum_{j_1~0}^{1-i_1}a_{j_1}c_{f_+}(2)^{1-i_1-j_1}b(i_1,0,0,4,28,16,16;2)
\end{align*}
where    $a_{0}=1$ and $a_1=8280$. Thus the assertion (1) can be verified by a computation with Mathematica.

(2) The assertion (2) follows from \cite[Table1]{Du}. (There are some errors  on that table, but  they are corrected  in www.neil-dummigan.staff.shef.ac.jp.)

(3) We have
$$L_{alg}(15,14:f)L_{alg}(25,24;f_+)=L_{alg}(25,14,f_+)L_{alg}(24,15,f_+).$$
Then the assertion (3) can be confirmed using Proposition \ref{prop.product-Hecke-L}.
\end{proof}
\begin{prop} 
\label{prop.numerical-example-16-2}
Let $h_+,g_{16}$ and  $A_0$ be as  above. Then 
\begin{align*}
&|c_{\widetilde {\mathcal F}_{h_+,g_{16}}}(A)|^2 L_{alg}(11, \widetilde {\mathcal F}_{h_+,g_{16}} , \mathrm {St})\\
&=\dfrac{2^{49} \cdot 34862967 (-222920204581 + 1281418453 \lambda)}{633217975}
.\end{align*}
\end{prop}
\begin{proof}
We note that 
\[|c_{\widetilde {\mathcal F}_{h_+,g_{16}}}(A)|^2 L_{alg}(11, \widetilde {\mathcal F}_{h_+,g_{16}} , \mathrm {St})=L(11,f_1,\mathrm{St}).\]
Therefore, by Theorem \ref{thm.special-pullback} we have
\[|c_{\widetilde {\mathcal F}_{h_+,g_{16}}}(A)|^2 L_{alg}(11, \widetilde {\mathcal F}_{h_+,g_{16}} , \mathrm {St})=\frac{\begin{vmatrix} C(14;A,A) & 1 & 1 \\ C(14;A_1,A) & c_{f_2}(A_1) & c_{f_3}(A_1) \\ C(14;A_2,A) & c_{f_2}(A_2) & c_{f_3}(A_2) \end{vmatrix}}{\begin{vmatrix} 1 & 1 & 1 \\ c_{f_1}(A_1) & c_{f_2}(A_1) & c_{f_3}(A_1) \\ c_{
f_1}(A_2) & c_{f_2}(A_2) & c_{f_3}(A_2) \end{vmatrix}}.\]
Here for $A, B \in \mathcal H_3(\mathbb Z)_{>0}$, $C(k;B,A)$ is that defined in Theorem \ref{thm.special-pullback}.
Thus the assertion can be verified  by a computation with Mathematica.
\end{proof}
By an argument similar to that in Theorem 
\ref{thm.main-example}, we can deduce the following theorem.
\begin{thm}
\label{thm.main-example16}
Let $h_+$ and $g_{16}$ be as above. Then, $C_{h_+,g_{16}}=2^{18}$, that is, Conjecture \ref{conj.A'} holds true for $h_+,g_{16}$ and so does Conjecture \ref{conj.A}.
\end{thm}

{\bf Corrections to \cite{I-K-P-Y}}

* p. 165, l. -4: For `$P_{r,\nu_1,\nu_2,\nu_3}(*)$', read `$P_{r-2,\nu_1,\nu_2,\nu_3}(*)$'.

* p. 166, l. 15: For `$c_{3,k_1,k_2,k_3}(*)$', read

 `$c_{3,(k_2+k_3-k_1-r)/2,,(k_1+k_3-k_2-r)/2,(k_1+k_2-k_3-r)/2}(*)$'.

\vskip 0.2cm

* p. 175, Table 6: The $T_7$-th Fourier coefficient of $f_3$ should be $-40$.

\end{document}